\documentclass[11pt]{article}
\usepackage{amsmath}
\usepackage{dsfont}
\usepackage{mathrsfs}
\usepackage{amsmath,amssymb}
\usepackage{amsfonts}
\usepackage{hyperref}
\usepackage{amsthm}
\usepackage{graphicx}
\usepackage{geometry}
\usepackage{subfigure}
\usepackage{caption}
\usepackage{enumerate}
\usepackage{xcolor}
\usepackage{cite}
\usepackage{epic,eepic,epsf,epsfig}
\usepackage{color}
\usepackage{floatflt}
\usepackage{float}
\usepackage{amssymb}
\usepackage{amsmath,amssymb,amsfonts}
\usepackage[utf8]{inputenc}
\usepackage{array}
\usepackage{color}

\usepackage{indentfirst}
\graphicspath{{figures/}}
\hfuzz=\maxdimen
\theoremstyle{definition}
\newtheorem{lemma}{Lemma}[section]

\newtheorem{proposition}[lemma]{Proposition}
\newtheorem{theorem}[lemma]{Theorem}
\newtheorem{corollary}[lemma]{Corollary}

\newtheorem{remark}{Remark}

\numberwithin{equation}{section}

\DeclareFixedFont{\Acknowledgment}{OT1}{cmr}{bx}{n}{14pt}
\allowdisplaybreaks

\begin{document}
\title{\bf  Improved explicit estimates for the discrete Laplace operator with hyperbolic circle patterns}
\author{Aijin Lin, Longxiang Wu}
\date{}
\maketitle

\begin{abstract}

    \par\quad
    Ge in his thesis \cite{Ge-thesis} introduced the combinatorial Calabi flows and established the long time existence and convergence of solutions to the flows in both hyperbolic and Euclidean background geometries. It is noteworthy that the existence of solutions to the combinatorial Calabi flows in hyperbolic background geometry proves to be more intricate and challenging compared to the Euclidean background geometry. The main difficulty is to establish the compactness, especially the lower boundeness along the flow equations. In this paper, we give two explicit estimates for the discrete Laplace operator based on the Glickenstein-Thomas formulation \cite{Glickenstein2017} for discrete hyperbolic conformal structures. As applications, we give new proofs of the long time existence of solutions to the combinatorial Calabi flows established by Ge-Xu \cite{Ge2016}, Ge-Hua \cite{Ge2018} and the combinatorial $p$-th Calabi flows established by Lin-Zhang \cite{Lin2019} in hyperbolic background geometry.

    \par\quad
    \newline \textbf{Keywords}: Combinatorial $p$-th Calabi flows; Discrete Laplace operator; Long time existence

\end{abstract}

\vspace{12pt}

\section{Introduction}
Over the past four decades, geometric flows have become one of the most fruitful areas of geometric analysis. Hamilton's Ricci flow \cite{Hamilton1982} defined by the evolution equation \( \frac{\partial g}{\partial t} = -2 \text{Ric}(g) \), deforms the metric based on the Ricci curvature and is critical in resolving the Poincaré conjecture, the Thurston geometrization conjecture, and the differentiable sphere theorem. Calabi \cite{Calabi1982,Calabi1985} studied the variational problem of minimizing the so-called “Calabi energy” in any fixed cohomology class of K$\ddot{\text{a}}$hler metrics, defining the Calabi flow as the gradient flow of the Calabi energy. On two dimensional surfaces, both the normalized Ricci flow and the Calabi flow exist for all time and converge to a constant curvature scalar metric (c.f. \cite{Chang2000,Chang2006,Chen2001} for further references).

Thurston \cite{Thurston1976} introduced hyperbolic circle packings (also called circle patterns) on triangulated surfaces in his program for hyperbolizing three-manifolds. The singularity at each vertex is measured by the discrete Gaussian curvature \(K_i\), calculated as the difference of \(2\pi\) and the cone angles. Thurston established that the set of all possible values for these cone angles corresponds to a convex polytope and demonstrated that a circle packing is uniquely determined by its cone angles.

One can also study geometric flows on polyhedral manifolds. As a discrete analogue of the smooth Ricci flow, the combinatorial Ricci flows were introduced by Chow and Luo \cite{BennettChow2003} for circle packings on triangulated surfaces. They proved that the combinatorial Ricci flows exist for all time and converge exponentially fast to the Thurston's circle packings, which gave an new proof of the Koebe-Andreev-Thurston theorem. Since then, combinatorial curvature flows have played significant roles in the study of low dimensional topology and practical applications. For further reading, we refer to the work of Beardon-Stephenson \cite{Beardon1990}, He-Schramm \cite{He1995}, Luo \cite{Luo2004,Luo2005}, Guo \cite{Guo2011}, Ge \cite{Ge-thesis,Ge2017a}, Ge-Xu \cite{Ge2016}, Ge-Jiang \cite{Ge2017, Ge2019, Ge2019a}, Ge-Hua \cite{Ge2018}, Feng-Ge-Hua \cite{Feng22}, Ge-Lin \cite{Ge2024}, Angle et al. \cite{Angle2016}, Glickenstein-Thomas \cite{Glickenstein2017}, Xu \cite{Xu2018}, Zhou \cite{Zhou2019}, Lin-Zhang \cite{Lin2019,Lin2021}, Feng-Lin-Zhang \cite{Feng2020}, Zhang-Zheng \cite{Zhang24}.

Especially, inspired by the work of Chow-Luo \cite{BennettChow2003}, Ge \cite{Ge-thesis,Ge2017a} introduced the combinatorial Calabi flows for Thurston's circle packings, defined as the negative gradient flow of the combinatorial Calabi energy.
Analogous to the combinatorial Ricci flows, Ge \cite{Ge-thesis,Ge2017a}, Ge-Xu \cite{Ge2016} and Ge-Hua \cite{Ge2018} established the long-time existence and exponential convergence of the combinatorial Calabi flows to the Thurston's circle packings on surfaces.
Lin and Zhang \cite{Lin2019} subsequently generalized Ge's pioneering work on the combinatorial Calabi flows from $p=2$ to any $p>1$. Xu-Wu \cite{Xu2021} introduced the fractional combinatorial Calabi flows which unify and extend the Chow-Luo's combinatorial Ricci flows, Luo's combinatorial Yamabe flows and Ge's combinatorial Calabi flows.
Unlike the combinatorial $p$-th Calabi flows, the fractional combinatorial Calabi flows reduces to the combinatorial Ricci flows for $s=0$ and to the combinatorial Calabi flows for $s=1$.
Recently, Hu et al. \cite{Hu2024} considered the combinatorial $p$-th Calabi flows in the setting of total geodesic curvatures.
The properties of the discrete Laplacian operator are of significant importance in the study of the combinatorial Calabi flows.

In this paper, we establish two important estimates for the discrete Laplacian operator in hyperbolic background geometry. The purpose of building these estimates is two-fold:
\begin{enumerate}
    \item  The first estimate is crucial for the proofs of Xu-Wu's main results \cite{Xu2021} on the combinatorial fractional Calabi flows.
    \item  The second estimate extends the Ge-Xu's key estimate \cite{Ge2016} for the discrete Laplacian operator from the zero weight $\Phi\equiv0$ to any weight $\Phi \in [0,\pi)$, which in turn provides new proofs of the long time existences of Ge's combinatorial Calabi flows \cite{Ge2018} and Lin-Zhang's the combinatorial $p$-th Calabi flows \cite{Lin2019} in hyperbolic background geometry.
\end{enumerate}

\subsection{Circle packings with obtuse intersection angles}
Let $M$ be a closed surface with a triangulation $\mathcal{T}=(V,E,F)$, where $V, E, F$ denote the sets of all vertices, edges, and faces respectively.
A \emph{weight} on a triangulation $\mathcal{T}$ is defined to be a function $\Phi: E \to [0, \pi)$. The triple $(M, \mathcal{T}, \Phi)$ will be referred to as a weighted triangulation of $M$.

Throughout this paper, we focus on the hyperbolic background geometry denoted by $\mathbb{H}^2$. A function defined on the vertex set $V$ will be identified with an $N$-dimensional column vector, where $N = |V|$ denotes the number of vertices.
Moreover, we fix an ordering for all vertices, labeled $v_1, v_2, \ldots, v_N$. Let $c_i$ denote the circle associated with each vertex $v_i \in V$ and $r_i\in (0, +\infty)$ be the radius of circle $c_i$. We call the radius function $r: V \to (0, +\infty)^{N}=\mathbb{R}_{>0}^{N}$ a \emph{circle packing metric}.

Given a fixed triangulated surface $(M, \mathcal{T}, \Phi)$,
a circle packing metric $r$ determines a piecewise linear metric on $M$. Specifically, in $\mathbb{H}^2$
the length of an edge $e_{ij} \in E$ under this metric is given by
$$l_{ij}=\cosh^{-1}(\cosh r_i\cosh r_j+\sinh r_i\sinh r_j\cos(\Phi_{ij})),$$
where $\Phi_{ij}=\Phi(e_{ij})$ denotes the weight for the edge $e_{ij} \in E$. In the case of $\Phi \in [0, \pi/2]$, Thurston \cite{Thurston1976} first showed the three-circle configuration theorem, which states that for each face $\triangle_{ijk} \in F$, the edge lengths $l_{ij}, l_{jk}, l_{ik}$ satisfy the triangle inequalities.
Consequently, every face in the set $F$ is isometric to a hyperbolic triangle.
When $\Phi \in [0,\pi)$, one can construct examples of three circles such that the edge lengths formed by their configuration do not satisfy the triangle inequalities.
However, Zhou \cite{Zhou2019} proved that Thurston's three-circle configuration theorem remains valid under the condition $\langle\star\rangle$ below.
Furthermore, the triangulated surface $(X, \mathcal{T}, \Phi)$ is constructed by coherently gluing hyperbolic triangles.
Let $I_{ij} = \cos \Phi_{ij}$ for each edge $e_{ij} \in E$.
The condition $\langle\star\rangle$ is as follows:
\begin{equation*}
    I_{i j}+I_{i k} I_{j k} \geq 0, I_{i k}+I_{i j} I_{j k} \geq 0, I_{j k}+I_{i j} I_{i k} \geq 0, \forall \triangle_{ijk} \in F.  \tag*{$\langle\star\rangle$}
\end{equation*}
Consider a triangulated surface $(X,\mathcal{T},\Phi)$ with a circle packing metric $r=(r_1, \dots, r_N)$.
Let $\theta_i^{jk}$ denote the interior angle at vertex $v_i$ in the triangle $\triangle_{ ijk} \in F$,
we have $\theta_{i}^{jk} \in (0,\pi).$
The well-known combinatorial Gauss curvature $K_i$ at vertex $v_i$ is defined as
\begin{equation}
    \label{Equation:Combinatorial_Gauss_curvature}
    K_i = 2\pi - \sum_{\triangle_{ijk} \in F} \theta_{i}^{jk},
\end{equation}
where the sum is taken over each triangle that has vertex $ v_i $ as one of its vertices.

Motivated by the construction of the Calabi energy in the smooth setting, Ge \cite{Ge-thesis,Ge2017a} introduced the combinatorial Calabi energy as follows:
\begin{equation}
    \label{Equation:Calabi_energy}
    \mathcal{C}(r) = ||K||^2 = \sum_{i=1}^N K_i^2.
\end{equation}
For convenience sake, we choose the coordinate transformation $u_i = \ln \tanh (r_i /2)$ in $\mathbb{H}^2$. The map $u = u(r)$ is a diffemorphism between $\mathbb{R}^N_{>0}$ and $\mathbb{R}^N_{<0}$.
Let \(K(u) = (K_1, \dots, K_N)^T\) denote the curvature map.
The discrete dual-Laplacian denoted by
\(\Delta\), which is a special type of the discrete Laplacian, is defined as $\Delta=-L^{T}$, where $L$ is the negative Jacobian of the curvature map $K(u)$ with respect to \(u = (u_1, \dots, u_N)\). Thus we have
\begin{equation}
    \label{Equation:Jacobian_matrix}
    \Delta=-L^{T} = -\frac{\partial(K_1, \dots, K_N)}{\partial(u_1, \dots, u_N)}.
\end{equation}
Both \(\Delta\) and $L^{T}$ operate on functions $f$ (defined on the set $V$ of vertices, hence is a column
vector) by a matrix multiplication, i.e.
\begin{equation}\label{discrete dual-Laplacian}
  \Delta f_i=(\Delta f)_i=-(L^{T} f)_i=-\sum_{j=1}^{N}\frac{\partial K_{j}}{\partial u_i}f_j, 1\leq i\leq N.
\end{equation}
The operator \(\Delta\) in (\ref{discrete dual-Laplacian}) is called the \emph{discrete Laplacian operator}. Colin de Verdi$\grave{\textrm{e}}$re \cite{Verdiere1991} first proved
\begin{equation}\label{symmetry}
  \frac{\partial K_i}{\partial u_j}=\frac{\partial K_j}{\partial u_i}.
\end{equation}
in the case where the weight $\Phi\equiv0$. Chow–Luo \cite{BennettChow2003} further generalized this result to any weight $\Phi\in[0, \pi)$. The gradient of the discrete Calabi energy \(\mathcal{C}\) (\ref{Equation:Calabi_energy}), taken with respect to the \(u\)-coordinates, is then expressed as
\begin{equation*}
    \nabla_u \mathcal{C} = (\nabla_{u_1} \mathcal{C}, \dots, \nabla_{u_N} \mathcal{C})^T = -2 \Delta K = -2(\Delta K_1, \dots, \Delta K_N).
\end{equation*}
The \emph{combinatorial Calabi flow} is defined as the negative gradient flow of the discrete Calabi energy
\begin{equation}
    \label{Equation:Combinatorial_Calabi_flow}
    u^{\prime}(t)=\Delta K=-\frac12\nabla_{u}\mathcal{C}.
\end{equation}
Denote $v_j \sim v_i$ if the vertices $v_i$ and $v_j$ are adjacent (i.e., $\exists \ e_{ij} \in E$). For any vertex $v_i$ and any edge $e_{ij} \in E$, set
\begin{equation}\label{AB-equality}
    \begin{gathered}
        B_{ij} = \frac{\partial(\theta_i^{jk}+\theta_i^{jl})}{\partial r_j}\sinh r_j=\frac{\partial\theta_i^{jk}}{\partial u_j}+\frac{\partial\theta_i^{jl}}{\partial u_j},\\
        A_i= \sinh r_i\frac\partial{\partial r_i}\Big(\sum_{\triangle_{ijk} \in F}\operatorname{Area}(\triangle_{ijk})\Big)=\frac\partial{\partial u_i}\Big(\sum_{\triangle_{ijk} \in F}\operatorname{Area}(\triangle_{ijk})\Big),
    \end{gathered}
\end{equation}
where $v_k, v_l$ are the vertices such that $\triangle_{ijk}$ and $\triangle_{ijl}$ are adjacent faces. Ge-Xu \cite{Ge2016}, Ge-Hua \cite{Ge2018} showed that \begin{eqnarray}\label{AB-inequality}
\frac{\partial K_{i}}{\partial u_j}=\left\{
\begin{aligned}
&A_i+\sum_{v_k\sim v_i}B_{ik}, & v_j=v_i, \\
&-B_{ij}, & v_j\sim v_i, \\
&0, &\text{else}.
\end{aligned}
\right.
\end{eqnarray}
Hence the discrete Laplacian operator \(\Delta\) can be rewritten as
\begin{equation}\label{discrete Laplacian operator}
  \Delta f_i= \sum_{v_j \sim v_i} B_{ij} (f_j - f_i) - A_i f_i.
\end{equation}
and the combinatorial Calabi flow (\ref{Equation:Combinatorial_Calabi_flow}) in hyperbolic background geometry can be rewritten as
\begin{equation}
\label{Equation:rewritten_calabi_equation}
    \begin{aligned}
        u_{i}^{\prime}(t) = \Delta K_i = \sum_{v_j \sim v_i} B_{ij} (K_j - K_i) - A_i K_i.
    \end{aligned}
\end{equation}
Furthermore, for any $p>1$, Lin-Zhang \cite{Lin2019} generalized the discrete Laplacian operator \(\Delta\) and defined the $p$-th discrete Laplacian operator \(\Delta_p\) which operates on functions $f: V\rightarrow \mathbb{R}$ as follows:
\begin{equation}\label{p-laplacian}
  \Delta_p f_i= \sum_{v_j \sim v_i} B_{ij} |f_j - f_i|^{p-2} (f_j - f_i) - A_i f_i.
\end{equation}
When $p=2$, the $p$-th discrete Laplacian operator \(\Delta_p\) (\ref{p-laplacian}) is exactly the discrete Laplacian operator \(\Delta\).
The combinatorial $p$-th Calabi flow in hyperbolic background geometry introduced by Lin-Zhang \cite{Lin2019} is defined as
\begin{equation}
    \label{Equation:combinatorial_p_th_calabi_flow}
    u_{i}^{\prime}(t) = \sum_{v_j \sim v_i} B_{ij} |K_j - K_i|^{p-2} (K_j - K_i) - A_i K_i,
\end{equation}
When $p=2$, the combinatorial $p$-th Calabi flow (\ref{Equation:combinatorial_p_th_calabi_flow}) is exactly the combinatorial Calabi flow (\ref{Equation:rewritten_calabi_equation}).

\subsection{Main results}
By the work of Ge \cite{Ge-thesis,Ge2017a}, Ge-Xu \cite{Ge2016} and Ge-Hua \cite{Ge2018} on the combinatorial Calabi flows, Lin-Zhang \cite{Lin2019} on the combinatorial $p$-th Calabi flows, Xu-Wu \cite{Xu2021} on the fractional combinatorial Calabi flows, the existence of solutions to combinatorial Calabi flows in hyperbolic background geometry proves to be more intricate and challenging compared to the Euclidean background geometry.
The main difficulty is to establish the compactness property, especially the lower boundeness of solutions along the flow equations.
The following estimate for the discrete Laplace operator is one of our main results, which can be used to prove the lower boundeness and further the long time existence of solutions to the fractional combinatorial Calabi flows.
\begin{theorem}\label{main result1}
    \emph{
        Given a closed triangulated surface $(M, \mathcal{T}, \Phi)$ with $\Phi : E\rightarrow [0, \pi)$ satisfying the condition $\langle\star\rangle$.
        If the circle packing metric $r$ has a uniformly positive lower bound, i.e., $r_i \geq R > 0$ ($R$ is a constant) for any vertex $v_i \in V$,
    then there exists positive constants $a_1(\Phi, R), a_2(\Phi, R), a_3(\Phi, R)$ depending on the weight $\Phi$ and $R>0$ such that $a_1 \leq A_i \leq a_2$ and $0 \leq B_{ij} \leq a_3$ for any vertex $v_i, v_j \in V$.}
\end{theorem}
Theorem \ref{main result1} is first proved by Xu-Wu \cite{Xu2021} and plays a key role in their main results on the fractional combinatorial Calabi flows. We will give a new proof of Theorem \ref{main result1}, which is more explicit and direct than Xu-Wu's proof.

Furthermore, we show that the positive upper bound constants $a_2(\Phi, R), a_3(\Phi, R)$ in Theorem \ref{main result1} do not depend on the constant $R>0$ in fact. Therefore, we can improve Theorem \ref{main result1} after removing the assumption $r_i \geq R > 0, \forall v_i \in V$ and state our second main result as follows.

\begin{theorem}\label{main result2}
    \emph{Given a closed triangulated surface $(M, \mathcal{T}, \Phi)$ with  $\Phi : E\rightarrow [0, \pi)$ satisfying the condition $\langle\star\rangle$ .
    Then for any the circle packing metric $r$, there exists positive constants $C_1(\Phi), C_2(\Phi)$ depending only on the weight $\Phi$
     such that $0< A_i \leq C_1$ and
     $0<B_{ij} \leq C_2$ for any vertex $v_i, v_j \in V$.   }
\end{theorem}
Theorem \ref{main result2} provides a uniform upper boundness for the discrete Laplace operator, thereby generalizing an important estimate proved by Ge-Xu \cite{Ge2016} in the specific case when \(\Phi \equiv 0\).

To establish a uniform upper boundness for solutions to the combinatorial Calabi flows, Ge-Hua \cite{Ge2018} analyzed the geometric evolution of terms in the governing formula under metric variation. They proved that for $\Phi \in [0,\pi/2]$, there exists a constant $C > 0$ such that if the metric components $c_i \geq C$ for all $v_i \in V$, then $\frac\partial{\partial r_i}(2\mathrm{Area}(\triangle_{ijk})+\theta_i^{jk})\geq0$. This inequality was subsequently extended by Zhang-Zheng \cite{Zhang24} to the broader range $\Phi \in [0,\pi)$. By (\ref{AB-equality}), this inequality implies a relationship between $A_i$ and $B_{ij}$. We offer an alternative proof for the quantitative relationship between $A_i$ and $B_{ij}$ based on the Glickenstein-Thomas formulation \cite{Glickenstein2017} for discrete hyperbolic conformal structures, which is formally stated in Lemma \ref{Lemma:Glickenstein-lemma} (see Section 2) whose analytical verification is provided in the Appendix.

As applications, we can give new proofs of the long time existences of solutions to the combinatorial Calabi flows established by Ge-Xu \cite{Ge2016}, Ge-Hua \cite{Ge2018} and the combinatorial $p$-th Calabi flows established by Lin-Zhang \cite{Lin2019} in hyperbolic background geometry.

\begin{theorem}\label{long time existence}
    For any initial circle packing metric $r(0) \in \mathbb{R}_{>0}^{N}$, the solution to both the combinatorial Calabi flow (\ref{Equation:rewritten_calabi_equation}) and the combinatorial $p$-th Calabi flow (\ref{Equation:combinatorial_p_th_calabi_flow}) in $\mathbb{H}^{2}$ exists for all time $t \in [0,+\infty)$.
\end{theorem}

The paper is organized as follows. In Section 2, we give some preliminaries on the discrete laplace operator. In Section 3, we prove Theorem \ref{main result1} based the quantitative relationship between $A_i$ and $B_{ij}$. In Section 4, we prove Theorem \ref{main result2} by using the method of proof by cases. In Section 5, we prove Theorem \ref{long time existence}.

\noindent{\bf Acknowledgments.}
The first author would like to thank professor Gang Tian for constant encouragement and support. Both authors would like to thank professor Huabin Ge for many useful conversations. The first author is supported by NSFC (No. 12171480), Natural Science Foundation of Hunan Province (No. 2022JJ10059) and Scientific Research Program of NUDT (No. JS2023-01).

\section{Preliminaries}

In this section, we give necessary preliminaries. Recall the discrete Laplacian operator \(\Delta\) acts on a function \( f \) via matrix multiplication:
\begin{equation*}
    \Delta f_i = -\sum_{j=1}^{N} \frac{\partial K_j}{\partial u_i} f_j = \sum_{v_j \sim v_i} B_{ij}(f_j - f_i) - A_i f_i.
\end{equation*}
It is worth noting that in Lemma A1 of Chow-Luo \cite{BennettChow2003}, the weight satisfies $\Phi \in [0,\pi)$; hence, under the obtuse weight condition (i.e., $\Phi \in [0,\pi)$) considered in this paper, the relationship between $B_{ij}$ and $A_i$ with respect to $L$ remains valid. Glickenstein-Thomas \cite{Glickenstein2017} obtained a more fundamental relationship between $B_{ij}$ and $A_i$ through the analysis of the discrete conformal structure using geometric methods. Here, we provide an analytic proof of their conclusion.

\begin{lemma}[\cite{Glickenstein2017}, Proposition 9]
    \label{Lemma:Glickenstein-lemma}
    In $ \mathbb{H}^{2} $, given a triangulated surface $ (M, \mathcal{T}, \Phi) $ with weight $ \Phi \in[0, \pi) $, for any triangle $ \triangle_{i j k} $, we have
    \begin{equation}
        \label{Equation:Glickenstein-lemma}
        \begin{aligned}
            \frac{\partial \theta_{i}^{j k}}{\partial u_{i}} & =-\cosh l_{i j} \frac{\partial \theta_{i}^{j k}}{\partial u_{j}}-\cosh l_{i k} \frac{\partial \theta_{i}^{j k}}{\partial u_{k}} \\
            & =-\cosh l_{i j} \frac{\partial \theta_{j}^{i k}}{\partial u_{i}}-\cosh l_{i k} \frac{\partial \theta_{k}^{i j}}{\partial u_{i}}
            \end{aligned}
    \end{equation}
\end{lemma}

\begin{remark}
    \label{Remark:Glickenstein-lemma}
    In \cite{Xu2018}, Xu obtained the inherent symmetry of the dual structure of circle packings to the equalities (\ref{Equation:Glickenstein-lemma-proof_4}) and (\ref{Equation:Glickenstein-lemma-proof_5}) (see Appendix) ,
    thereby obtaining the following expression
    \begin{equation}
        \label{Equation:Glickenstein-remark_1}
        \frac{\partial \theta_{i}^{jk}}{\partial u_j} = \frac{1}{A_{ijk} \sinh^2 l_{ij}} \left[ C_k S_i^2 S_j^2 \left(1 - \cos^2 \Phi_{ij}\right) + C_i S_i S_j^2 S_k \gamma_{jik} + C_j S_i^2 S_j S_k \gamma_{ijk} \right].
    \end{equation}
    Here, $\gamma_{ijk} = \cos \Phi_{jk} + \cos \Phi_{ij} \cos \Phi_{ik}.$
    Based on the symmetry present in the above expression, we deduce that
    \begin{equation}
        \label{Equation:Glickenstein-remark_2}
        \frac{\partial \theta_{i}^{jk}}{\partial u_j} = \frac{\partial \theta_{j}^{ik}}{\partial u_i}.
    \end{equation}
    This is precisely the symmetry got in the proof of Lemma \ref{Lemma:Glickenstein-lemma}.
\end{remark}

\begin{corollary}
    \label{Corollary:Glickenstein-corollary}
    Based on the area formula for a hyperbolic triangle
    $\operatorname{Area}\left(\triangle_{ijk}\right)=\pi-\theta_{i}^{jk}-\theta_{j}^{ik}-\theta_{k}^{ij},$
    we have
    \begin{equation}\label{deri-area}
            \frac{\partial}{\partial u_{i}} \operatorname{Area}\left(\triangle_{ijk}\right)
            =\frac{\partial \theta_{j}^{ik}}{\partial u_{i}}\left(\cosh l_{ij}-1\right)+\frac{\partial \theta_{k}^{ij}}{\partial u_{i}}\left(\cosh l_{ik}-1\right),
    \end{equation}
 and     \begin{equation}\label{A-equality}
            A_{i}=\sum_{v_j \sim v_i}B_{ij}\left(\cosh l_{ij}-1\right).
    \end{equation}
    \end{corollary}
   \begin{proof}
     The first equality (\ref{deri-area}) is obvious by Lemma \ref{Lemma:Glickenstein-lemma}. Combining (\ref{deri-area}) with the expressions for $B_{ij}$ and $A_{i}$ (\ref{AB-equality}), we obtain
    \begin{equation*}
        \begin{aligned}
            A_{i} &=\sum_{\triangle_{ijk}\in F}\frac{\partial}{\partial u_{i}}\operatorname{Area}\left(\triangle_{ijk}\right)\\[0.5em]
            &=\sum_{\triangle_{ijk}\in F}\left[\frac{\partial \theta_{j}^{ik}}{\partial u_{i}}\left(\cosh l_{ij}-1\right)+\frac{\partial \theta_{k}^{ij}}{\partial u_{i}}\left(\cosh l_{ik}-1\right)\right]\\[0.5em]
            &=\sum_{v_j \sim v_i}\left(\frac{\partial \theta_{j}^{ik}}{\partial u_{i}}+\frac{\partial \theta_{j}^{il}}{\partial u_{i}}\right)\left(\cosh l_{ij}-1\right)\\[0.5em]
            &=\sum_{v_j \sim v_i}B_{ij}\left(\cosh l_{ij}-1\right).
        \end{aligned}
    \end{equation*}
    Here, $\triangle_{ijk}$ and $\triangle_{ijl}$ denote the two triangles sharing the edge $e_{ij}$. The final equality follows from the fact that
    $\frac{\partial \theta_{j}^{ik}}{\partial u_{i}}=\frac{\partial \theta_{i}^{jk}}{\partial u_{j}}$ by Remark \ref{Remark:Glickenstein-lemma}.  \hfill $\square$
    \end{proof}
With Corollary \ref{Corollary:Glickenstein-corollary} in hand, we now consider the positive definiteness of the Jacobian matrix $L^{T}$ under the condition $\Phi \in [0,\pi)$. We first present the entries of the matrix $L^{T}$ as follows:
\begin{equation*}
    \begin{pmatrix}
    A_{1}+\sum\limits_{j\neq 1}B_{1j} & -B_{12} & \cdots  & -B_{1N} \\
    -B_{21} & A_{2}+\sum\limits_{j\neq 2}B_{2j}  & \cdots  & -B_{2N} \\
    \vdots & \vdots & \ddots  & \vdots \\
    -B_{1,N-1} & -B_{2,N-1}  & \cdots  & -B_{N-1,N} \\
    -B_{1N} & -B_{2N} & \cdots  & A_{N}+\sum\limits_{j\neq N}B_{Nj}
    \end{pmatrix}.
\end{equation*}
It can be observed that this is a diagonally dominant matrix. Moreover, according to Corollary \ref{Corollary:Glickenstein-corollary}, the diagonal entries can be expressed as multiples of $B_{ij}$; therefore, the sign of $B_{ij}$ plays a critical role in determining the positive definiteness of the matrix.

Based on this observation, Zhou \cite{Zhou2019} discussed the structural conditions for obtuse circle packings. In particular, for any triangle $\triangle_{ijk} \in F$, Zhou \cite{Zhou2019} showed that the following inequalities hold:
\begin{equation*}
    \cos \Phi_{ij}+\cos \Phi_{ik}\cos \Phi_{jk}\ge 0,\\cos \Phi_{ik}+\cos \Phi_{ij}\cos \Phi_{jk}\ge 0,\ \cos \Phi_{jk}+\cos \Phi_{ij}\cos \Phi_{ik}\ge 0.
\end{equation*}
Set $\gamma_{ijk}=\cos \Phi_{jk}+\cos \Phi_{ij}\cos \Phi_{ik}$, then the above conditions can be rewritten as
\begin{equation*}
    \label{Equation:Zhou-condition}
    \tag{$\star$}
    \ \gamma_{ijk}\ge 0,\ \gamma_{jki}\ge 0,\ \gamma_{kij}\ge 0, \forall \triangle_{ijk}\in F.
\end{equation*}
Given the $(\star)$ condition,
Xu \cite{Xu2018} proved that the positive definiteness under the inversive distance condition,
and the obtuse case considered herein is a subset of that scenario.
We now state the positive definiteness lemma.
\begin{lemma}
    \label{Lemma:Positive-definiteness-lemma}
    In $\mathbb{H}^{2}$, given $(M, \mathcal{T}, \Phi)$ and assuming that the weight $\Phi$ satisfies the $(\star)$ condition, then the Jacobian matrix $L^{T}=-\Delta=\frac{\partial(K_{1},\ldots,K_{N})}{\partial(u_{1},\ldots,u_{N})}$ is symmetric and positive definite.
\end{lemma}

Furthermore, the finite lower boundness of the Laplace operator plays key role in the proof of the long-time existence of the combinatorial Calabi flows (\ref{Equation:rewritten_calabi_equation}).
Ge-Xu first studied the finite upper boundness of the Laplace operator in $\mathbb{H}^2$ when
$ \Phi \equiv 0$ (see \cite{Ge2016}, Proposition 3.1).
\begin{proposition}
    \label{Proposition:Ge-Xu-upper-bound}
    In $\mathbb{H}^{2}$, given $(M, \mathcal{T}, \Phi \equiv 0)$, we have
    \begin{equation*}
        \begin{gathered}
        0 < B_{ij} < 1, \\
        0 < A_{i} < d_{i}\cosh 1 \leq d\cosh 1,
        \end{gathered}
    \end{equation*}
where $d = \max_{1 \leq i \leq N}\{ d_{i} \}$ and $d_{i}$ denotes the degree at vertex $i$.

\end{proposition}

\section{The first improved estimate on the discrete Laplace operator}

In this section, we establish the first improved explicit estimate on the Laplace operator (\ref{discrete Laplacian operator}) and prove Theorem \ref{main result1}. Based on the quantitative relationship (Corollary \ref{Corollary:Glickenstein-corollary})
between $A_i$ and $B_{ij}$ and their expressions (\ref{AB-equality}),
our proof strategy is to first analyze the properties of the partial derivative $\frac{\partial \theta_i}{\partial u_j}$,
and then consider its properties when these derivatives are summed.
Our crucial observation is that for every face $\triangle_{ijk} \in F$,
there exists an edge $e_{ij} \in E$ such that $B_{ij} > 0$.\par
Firstly, we prove the following inequality.
\begin{proposition}
    	\label{proof_middle_prop}
    	Given three constants \(x, y, z \in [c, 1] , c \in (-1, 0]\). Suppose that the following condition holds
    	\begin{equation}
            \label{Theorem_middle_prop_1_condition}
            x + yz \geq 0, \
            y + xz \geq 0, \
            z + xy \geq 0.
        \end{equation}
    	Then we have
    	\[
    	2 - x^2 - y^2 + x + yz + y + xz + z + xy \geq 1 + c.
    	\]    	
    \end{proposition}
\begin{proof}
    If $\max\{x, y\} \leq |c|$, using the condition (\ref{Theorem_middle_prop_1_condition}), we obtain
    \begin{equation*}
    	\begin{aligned}
            &2 - x^2 - y^2 + (x + yz) + (y + xz) + (z + xy)
            \\
            &\geq 2 - 2c^2 + 0 + 0 + 0 \\
                                                     &\geq 2(1+c)(1-c) \\
                                                     &> 1+c.
        \end{aligned}
    \end{equation*}

    If $\max\{x, y\} > |c|$, there are two cases. In the first case  $\min\{x,y\} \leq |c|$, using condition (\ref{Theorem_middle_prop_1_condition}) and $\max\{x,y\} \leq 1$, we obtain
    \begin{equation*}
    	2 - x^2 - y^2 + (x + yz) + (y + xz) + (z + xy) \geq 2 - 1 - c^2 + 0 + 0 + 0\geq 1+c.
    \end{equation*}
    In the second case $|c| < \min\{x,y\} \leq 1$, by direct calculation we have
    \begin{equation*}
        \begin{aligned}
            &x^2 + y^2 - xy \\
            &\leq x + y - xy \\
            &= 1 - (1-x)(1-y) \\
            &\leq 1.
        \end{aligned}
    \end{equation*}.
    Hence we obtain
    \begin{equation*}
    	\begin{aligned}
    		 &2 - x^2 - y^2 + (x + yz) + (y + xz) + (z + xy) \\
             & \geq 2 - (x^2 + y^2 - xy) + z + 0 + 0\\
    		 & \geq 1 + z \\
             & \geq 1 + c.
    	\end{aligned}
    \end{equation*}
    This completes the proof of the Proposition \ref{proof_middle_prop}. \hfill $\square$
\end{proof}

\begin{theorem}
    \label{Proposition:A_i_and_B_ij_bound_compactness}
    \emph{
    Given a closed triangulated surface $(M, \mathcal{T}, \Phi)$ with $\Phi : E\rightarrow [0, \pi)$ satisfying the condition $\langle\star\rangle$.
    If the circle packing metric $r$ has a uniformly positive lower bound, i.e., $r_i \geq R > 0$ ($R$ is a constant) for any vertex $v_i \in V$,
    then there exists positive constants $a_1(\Phi, R), a_2(\Phi, R), a_3(\Phi, R)$ depending on the weight $\Phi$ and $R>0$ such that $a_1 \leq A_i \leq a_2$ and
    $0 \leq B_{ij} \leq a_3$ for any vertex $v_i, v_j \in V$.}
\end{theorem}

\noindent\begin{proof}

   First by (\ref{Equation:Glickenstein-remark_1}), we have
    \begin{equation*}
        \frac{\partial{\theta_i^{jk}}}{\partial{u_j}} = \frac{\partial{\theta_j^{ik}}}{\partial{u_i}}
        = \frac{1}{A_{ijk}\sinh^2{l_{ij}}}[C_kS_i^2S_j^2(1 - {\cos^2{\Phi_{ij}}})
        + C_iS_iS_j^2S_k\gamma_{jik} + C_jS_i^2S_jS_k\gamma_{ijk}],
    \end{equation*}
    where $C_i = \cosh r_i, S_i = \sinh r_i, A_{ijk} = \sinh{l_{ij}}\sinh{l_{ik}}\sin{\theta_i^{jk}}$ and
    $\gamma_{ijk} = \cos{\Phi_{jk}} + \cos{\Phi_{ij}}\cos{\Phi_{ik}}$.
    Since the weights satisfy the condition $\langle\star\rangle$
    (i.e., $\forall \triangle_{ijk} \in F, \, \gamma_{ijk} \geq 0 , \gamma_{jik} \geq 0 , \gamma_{kij} \geq 0$) and the fact that $l_{ij}, l_{ik}, l_{jk} > 0$ and $\theta_{i}^{jk} \in (0,\pi)$,
    this immediately implies
    \begin{equation*}
        \frac{\partial \theta_{i}^{jk}}{\partial u_j} \geq 0,
    \end{equation*}
    where the equality holds when $\Phi_{ij}=0$ and $\Phi_{ik}+\Phi_{jk}=\pi$.
    If $\triangle_{ijk} \in F$ and $\triangle_{ijl} \in F$ are adjacent faces,
    then $B_{ij} = \partial \theta_{i}^{jk} / \partial u_j + \partial \theta_{i}^{jl} / \partial u_j \geq 0$, where the equality holds when
    $\Phi_{ij}=0$, $\Phi_{ik}+\Phi_{jk}=\pi$, and $\Phi_{il}+\Phi_{jl}=\pi$.

    Set
    \begin{equation*}
        \cos{\Phi_{ij}} = \varphi_{ij},\ \cos{\Phi_{jk}} = \varphi_{jk},\ \cos{\Phi_{ik}} = \varphi_{ik}.
    \end{equation*}
    By calculation we have
    \begin{equation}
        \label{Equation:compact_proof_basic_expression_total}
        \frac{\partial{\theta_i^{jk}}}{\partial{u_j}} = \frac{\partial{\theta_j^{ik}}}
        {\partial{u_i}}
        =\frac{C_kS_i^2S_j^2(1-\varphi_{ij}^2) + C_iS_iS_j^2S_k(\varphi_{ik} + \varphi_{ij}\varphi_{jk}) + C_jS_i^2S_jS_k(\varphi_{jk} + \varphi_{ij}\varphi_{ik})}
        {[(C_iC_j + \varphi_{ij}S_iS_j)^2 - 1]\sqrt{\Delta}},
    \end{equation}
    where
    \begin{equation}
        \label{Equation:compact_proof_basic_expression_delta}
        \begin{aligned}
            \Delta
            &= (2+2\varphi_{ij}\varphi_{jk}\varphi_{ik})S_i^2S_j^2S_k^2 + (1-\varphi_{ij}^2)S_i^2S_j^2 + (1-\varphi_{jk}^2)S_j^2S_k^2
            + (1-\varphi_{ik}^2)S_i^2S_k^2 \\[1mm]
            &+ (2\varphi_{ij} + 2\varphi_{jk}\varphi_{ik})C_iC_jS_iS_jS_k^2 + (2\varphi_{jk} + 2\varphi_{ij}\varphi_{ik})C_jC_kS_jS_kS_i^2
            + (2\varphi_{ik} + 2\varphi_{ij}\varphi_{jk})C_iC_kS_iS_kS_j^2.
        \end{aligned}
    \end{equation}

    Given that $r_i \geq R > 0$ for every $v_i \in V$,
    and $S_i = \sinh{r_i} = \frac{1-e^{-2r_i}}{2}e^{r_i}$ and $C_i = \cosh{r_i} = \frac{1+e^{-2r_i}}{2}e^{r_i}$,
    it follows that there exists a positive constant
    $a(R) = \frac{1 - e^{-R}}{2}$ such that
    \begin{equation*}
    	2a^2 = \frac{1 + e^{-2R} - 2e^{-R}}{2} < \frac{1 - e^{-2R}}{2}.
    \end{equation*}
    Since $\frac{1 - e^{-2r_i}}{2} \geq \frac{1 - e^{-2R}}{2} > 2a^2$, we have
    \begin{equation}
        \label{Equation:compact_proof_sinh_bound}
    	2a^2e^{r_i}  \leq S_i \leq e^{r_i}  \ \text{for} \ i =1,2,\cdots, N.
    \end{equation}
	By the definition of the hyperbolic sine function, we have
	\begin{equation*}
		C_i = 2 \sinh^2 \frac{r_i}{2} + 1 = 2 (\frac{1 - e^{-r_i}}{2})^2 e^{{r_i}} + 1 \geq 2a^2e^{r_i} + 1,
	\end{equation*}
	hence we obtain
	\begin{equation}
        \label{Equation:compact_proof_cosh_bound}
	 2a^2e^{r_i} + 1 \leq C_i \leq e^{r_i}\ \text{for} \ i =1,2,\cdots, N.
	\end{equation}

    Now, we establish an upper and a lower bounds estimate for $\cosh l_{ij}-1$. We begin by expanding its expression as follows.
    \begin{equation*}
    	\begin{aligned}
    		\cosh{l_{ij}} - 1
    		& = C_iC_j + \varphi_{ij}S_iS_j - 1 \\[1mm]
    		& = \frac{e^{r_i} + e^{-r_i}}{2}\frac{e^{r_j} + e^{-r_j}}{2} + \varphi_{ij} \frac{e^{r_i} - e^{-r_i}}{2}\frac{e^{r_j} - e^{-r_j}}{2} -1 \\[1mm]
    		& = \frac{(1+\varphi_{ij})(e^{r_i+r_j}+e^{-r_i-r_j}) + (1-\varphi_{ij})(e^{-r_i+r_j}+e^{r_i-r_j})}{4} -1 \\[1mm]
    		& = e^{r_i+r_j}\frac{(1+\varphi_{ij})(1+e^{-2r_i-2r_j}) + (1-\varphi_{ij})(e^{-2r_i}+e^{-2r_j} )- 4 e^{-r_i-r_j}}{4}.
    	\end{aligned}
    \end{equation*}
    Let $\bar{\varphi} = \min\{\,0, \min\limits_{e_{ij}\in E} \varphi_{ij}\}\,$.
    By (\ref{Equation:compact_proof_sinh_bound}) and (\ref{Equation:compact_proof_cosh_bound}),
    we obtain the lower bound estimate
    \begin{equation}
        \label{Equation:compact_proof_cosh_lower_bound}
    	\begin{split}
    		\cosh{l_{ij}} - 1
            & \geq(2a^2e^{r_i} + 1) (2a^2e^{r_j} + 1) + 4\varphi_{ij}a^{4}e^{r_i+r_j} - 1 \\[1mm]
            & = 4a^4 (1 + \varphi_{ij}) e^{r_i + r_j} + 2a^2 e^{r_i} + 2a^2 e^{r_j}\\[1mm]
    		& \geq 4a^4 (1 + \bar\varphi) e^{r_i + r_j} \triangleq m_1e^{r_i + r_j},
    	\end{split}
    \end{equation}
     and the upper bound estimate
    \begin{equation}
        \label{Equation:compact_proof_cosh_upper_bound}
        \begin{split}
            \cosh{l_{ij}} - 1
            & = C_i C_j + \varphi_{ij} S_i S_j - 1\\[1mm]
            & \leq e^{r_i+r_j} + e^{r_i+r_j} - 1 \\[1mm]
            & \leq 2e^{r_i+r_j} \triangleq m_2 e^{r_i+r_j}.
    \end{split}
    \end{equation}
    In summary, we obtain two positive constants
    \begin{equation*}
        m_1(R, \Phi) = 4a^4(1+\bar{\varphi}) \ \text{ and } \ m_2 = 2.
    \end{equation*}

    Now we provide an upper and a lower bound estimate for the numerator of (\ref{Equation:compact_proof_basic_expression_total}).
    Combining the estimates (\ref{Equation:compact_proof_sinh_bound}) and (\ref{Equation:compact_proof_cosh_bound}) with the condition $\varphi \in (-1,1]$ for any $e_{ij} \in E$, we obtain the upper bound
    \begin{equation}
        \label{Equation:compact_proof_numerator_upper_bound}
        \begin{split}
            & C_kS_i^2S_j^2(1 - \varphi_{ij}^2) + C_iS_iS_j^2S_k(\varphi_{ik} + \varphi_{ij}\varphi_{jk}) + C_jS_i^2S_jS_k(\varphi_{jk} + \varphi_{ij}\varphi_{ik}) \\[1mm]
            &\leq [(1 - \varphi_{ij}^2) +
            (\varphi_{ik} + \varphi_{ij}\varphi_{jk}) +
            (\varphi_{jk} + \varphi_{ij}\varphi_{ik})]e^{2r_i + 2r_j + r_k} \\[1mm]
            & \leq (5-\varphi_{ij}^2)e^{2r_i + 2r_j + r_k}\\[1mm]
            &\leq 5e^{2r_i + 2r_j + r_k} \triangleq m_3e^{2r_i + 2r_j + r_k}.
        \end{split}
    \end{equation}
    By the condition $\langle\star\rangle$ (i.e., for any $\triangle_{ijk}\in F$, we have $\varphi_{ij}+\varphi_{ik}\varphi_{jk}\geq 0$) and the two estimates (\ref{Equation:compact_proof_sinh_bound}) and (\ref{Equation:compact_proof_cosh_bound}),
    we obtain the lower bound
    \begin{equation}
        \label{Equation:compact_proof_numerator_lower_bound}
        \begin{split}
            & C_kS_i^2S_j^2(1 - \varphi_{ij}^2) + C_iS_iS_j^2S_k(\varphi_{ik} + \varphi_{ij}\varphi_{jk}) + C_jS_i^2S_jS_k(\varphi_{jk} + \varphi_{ij}\varphi_{ik}) \\
            &\geq 32a^{10}[(1 - \varphi_{ij}^2) +
            (\varphi_{ik} + \varphi_{ij}\varphi_{jk}) +
            (\varphi_{jk} + \varphi_{ij}\varphi_{ik})]e^{2r_i + 2r_j + r_k} \\
            &\triangleq m_4e^{2r_i + 2r_j + r_k}.
        \end{split}
    \end{equation}
    Therefore, we obtain another two positive constants
    \begin{equation*}
        m_3=5  \ \text{ and } \ m_4(R, \Phi) =  32a^{10}[(1 - \varphi_{ij}^2) +
        (\varphi_{ik} + \varphi_{ij}\varphi_{jk}) +
        (\varphi_{jk} + \varphi_{ij}\varphi_{ik})].
    \end{equation*}

    Next, we provide an upper and a lower bound estimate for the denominator of (\ref{Equation:compact_proof_basic_expression_total}), which consists of two terms.
    We begin with the upper bound of the first term $(C_iC_j + \varphi_{ij}S_iS_j)^2 - 1$ as follows.
    \begin{equation}
        \label{Equation:compact_proof_denominator_first_part_upper_bound}
		\begin{split}
			&(C_iC_j+\varphi_{ij}S_iS_j)^2-1\\[1mm]
			&=(C_iC_j+\varphi_{ij}S_iS_j+1)(C_iC_j+\varphi_{ij}S_iS_j-1)\\[1mm]
			&\leq (e^{r_i+r_j} + e^{r_i+r_j} + 1) (e^{r_i+r_j} + e^{r_i+r_j} - 1)\\[1mm]
			&\leq (2e^{r_i + r_j} + 1)(2e^{r_i + r_j})\\[1mm]
			&\leq 6e^{2r_i+2r_j}\triangleq m_5e^{2r_i+2r_j}.
		\end{split}
	\end{equation}
    For the first term, we obtain a lower bound estimate
    \begin{equation}
        \label{Equation:compact_proof_denominator_first_part_lower_bound}
        \begin{split}
        	&(C_iC_j+\varphi_{ij}S_iS_j)^2-1\\[1mm]
        	& = (C_i C_j + \varphi_{ij} S_i S_j + 1) (C_i C_j + \varphi_{ij} S_i S_j - 1)\\[1mm]
        	& \geq ((2a^2e^{r_i}+1)(2a^2e^{r_j}+1) + 4a^4 \varphi_{ij} e^{r_i + r_j} + 1)((2a^2e^{r_i}+1)(2a^2e^{r_j}+1) + 4a^4 \varphi_{ij} e^{r_i + r_j} - 1)\\[1mm]
        	&=(4a^4 (1 + \varphi_{ij}) e^{r_i + r_j} + 2a^2 e^{r_i} + 2a^2 e^{r_j} + 2) (4a^4 (1 + \varphi_{ij}) e^{r_i + r_j} + 2a^2 e^{r_i} + 2a^2 e^{r_j} )\\[1mm]
        	&\geq 16a^8(1+\bar\varphi)^2e^{2r_i+2r_j}\triangleq m_6e^{2r_i+2r_j},
        \end{split}
    \end{equation}
    where we used the two estimates (\ref{Equation:compact_proof_sinh_bound}) and (\ref{Equation:compact_proof_cosh_bound}) in the third line.
    Hence, we obtain
    \begin{equation*}
        m_5 = 6 \ \text{ and } \ m_6(R,\Phi) =  16a^8(1+\bar\varphi)^2.
    \end{equation*}

    For the second part $\sqrt{\Delta}$, first we have the following upper bound estimate
        \begin{equation}
        \label{Equation:compact_proof_denominator_second_part_upper_bound}
        \begin{aligned}
            \Delta
            &= (2+2\varphi_{ij}\varphi_{jk}\varphi_{ik})S_i^2S_j^2S_k^2 + (1-\varphi_{ij}^2)S_i^2S_j^2 + (1-\varphi_{jk}^2)S_j^2S_k^2
            + (1-\varphi_{ik}^2)S_i^2S_k^2 \\[1mm]
            &\quad + (2\varphi_{ij} + 2\varphi_{jk}\varphi_{ik})C_iC_jS_iS_jS_k^2 + (2\varphi_{jk} + 2\varphi_{ij}\varphi_{ik})C_jC_kS_jS_kS_i^2
            + (2\varphi_{ik} + 2\varphi_{ij}\varphi_{jk})C_iC_kS_iS_kS_j^2 \\[1mm]
            &\leq (2+2\varphi_{ij}\varphi_{jk}\varphi_{ik})e^{2r_i + 2r_j + 2r_k}+(1-\varphi_{ij}^2)e^{2r_i + 2r_j}
            + (1-\varphi_{jk}^2)e^{2r_j + 2r_k} + (1-\varphi_{ik}^2)e^{2r_i + 2r_k} \\[1mm]
            &\quad + (2\varphi_{ij} + 2\varphi_{jk}\varphi_{ik})e^{2r_i + 2r_j + 2r_k}
            +(2\varphi_{jk} + 2\varphi_{ij}\varphi_{ik})e^{2r_i + 2r_j + 2r_k}
            +(2\varphi_{ik} + 2\varphi_{ij}\varphi_{jk})e^{2r_i + 2r_j + 2r_k} \\[1mm]
            &\leq 4e^{2r_i + 2r_j + 2r_k}  + (3 - \varphi_{ij}^2 - \varphi_{jk}^2 - \varphi_{ik}^2)e^{2r_i + 2r_j + 2r_k} + 4 e^{2r_i + 2r_j + 2r_k} + 4e^{2r_i + 2r_j + 2r_k} + 4e^{2r_i + 2r_j + 2r_k}\\[1mm]
            &\leq (19- \varphi_{ij}^2 - \varphi_{jk}^2 - \varphi_{ik}^2)e^{2r_i + 2r_j + 2r_k} \\[1mm]
            &\leq 19e^{2r_i + 2r_j + 2r_k} \triangleq m_7e^{2r_i + 2r_j + 2r_k}.
        \end{aligned}
    \end{equation}
    Then we give the lower bound estimate as follows.
    \begin{equation}
        \label{Equation:compact_proof_denominator_second_part_lower_bound}
        \begin{aligned}
            \Delta
            &= (2+2\varphi_{ij}\varphi_{jk}\varphi_{ik})S_i^2S_j^2S_k^2 + (1-\varphi_{ij}^2)S_i^2S_j^2 + (1-\varphi_{jk}^2)S_j^2S_k^2
            + (1-\varphi_{ik}^2)S_i^2S_k^2 \\[1mm]
            &\quad + (2\varphi_{ij} + 2\varphi_{jk}\varphi_{ik})C_iC_jS_iS_jS_k^2 + (2\varphi_{jk} + 2\varphi_{ij}\varphi_{ik})C_jC_kS_jS_kS_i^2
            + (2\varphi_{ik} + 2\varphi_{ij}\varphi_{jk})C_iC_kS_iS_kS_j^2 \\[1mm]
            &\geq 128(1+\varphi_{ij}\varphi_{jk}\varphi_{ik})a^{12}e^{2r_i + 2r_j + 2r_k} \\[1mm]
            &\quad +16(1-\varphi_{ij}^2)a^8e^{2r_i + 2r_j}
            + 16(1-\varphi_{jk}^2)a^8e^{2r_j + 2r_k} + 16(1-\varphi_{ik}^2)a^8e^{2r_i + 2r_k} \\[1mm]
            &\quad + 128(\varphi_{ij} + \varphi_{jk}\varphi_{ik} + \varphi_{jk} + \varphi_{ij}\varphi_{ik} + \varphi_{ik} + \varphi_{ij}\varphi_{jk}) a^{12} e^{2r_i + 2r_j + 2r_k} \\[1mm]
            &\geq 128(1+\varphi_{ij}\varphi_{jk}\varphi_{ik})a^{12}e^{2r_i + 2r_j + 2r_k} \\[1mm]
            &\geq 128(1+\bar{\varphi})a^{12}e^{2r_i + 2r_j + 2r_k} \triangleq m_8e^{2r_i + 2r_j + 2r_k}.
       \end{aligned}
    \end{equation}
    The first inequality follows from the estimates (\ref{Equation:compact_proof_sinh_bound}) and (\ref{Equation:compact_proof_cosh_bound}), together with the inequality $2a^2 e^{r_i} + 1 > 2a^2 e^{r_i}$.
    The section inequality follows from the range of $\varphi_{ij}$ and the condition $\langle \star \rangle$.
    The last inequality follows from the fact that for any $x,y,z \in (-1,1]$, $xyz \geq \min\{0,\min\{x,y,z\}\}$.

    Consequently, we obtain the fourth collection of positive constants
    \begin{equation*}
        m_7(R) = 19 \ \text{ and } \  m_8(R,\Phi) = 128(1+\bar{\varphi})a^{12}.
    \end{equation*}
	Thus, multipling the two estimates (\ref{Equation:compact_proof_denominator_first_part_upper_bound}) and (\ref{Equation:compact_proof_denominator_second_part_upper_bound}) gives
	\begin{equation}
        \label{Equation:compact_proof_denominator_upper_bound}
	    \begin{aligned}
            ((C_iC_j + \varphi_{ij}S_iS_j)^2 - 1)\sqrt{\Delta} &\leq m_5e^{2r_i+2r_j}\sqrt{m_7e^{2r_i+2r_j+2r_k}} \\[1mm]
            &= 6e^{2r_i+2r_j} \sqrt{19e^{2r_i+2r_j+2r_k}} \\[1mm]
            &\leq 30e^{3r_i+3r_j+r_k},
	    \end{aligned}
	\end{equation}	
	and multipling the two estimates (\ref{Equation:compact_proof_denominator_first_part_lower_bound}) and (\ref{Equation:compact_proof_denominator_second_part_lower_bound}) gives
	\begin{equation}
        \label{Equation:compact_proof_denominator_lower_bound}
		\begin{aligned}
			((C_iC_j + \varphi_{ij}S_iS_j)^2 - 1)\sqrt{\Delta} &\geq m_6e^{2r_i+2r_j}\sqrt{m_8e^{2r_i+2r_j+2r_k}} \\[1mm]
			&=16a^8(1+\bar{\varphi})^2e^{2r_i+2r_j}\sqrt{128(1+\bar{\varphi})a^{12}e^{2r_i+2r_j+2r_k}} \\[1mm]
			&\geq 128 a^{14}(1+\bar{\varphi})^2\sqrt{1+\bar{\varphi}}\ e^{3r_i+3r_j+r_k}.
		\end{aligned}
	\end{equation}
		
    Combining the upper bound estimate (\ref{Equation:compact_proof_numerator_upper_bound})
    of the numerator and the lower bound estimate (\ref{Equation:compact_proof_denominator_lower_bound}) of the denominator, we obtain an upper bound estimate of (\ref{Equation:compact_proof_basic_expression_total})
    \begin{equation}
    	\label{Equation_partial_partial_upper}
    	\begin{aligned}
    		\frac{\partial{\theta_{j}^{ik}}}{\partial{u_i}}
    		&\leq \frac{m_3e^{2r_i+2r_j+r_k}}{ 128 a^{14}(1+\bar{\varphi})^2\sqrt{1+\bar{\varphi}}e^{3r_i+3r_j+r_k}}  \\[1mm]
    		& =\frac{5}{ 128 a^{14}(1+\bar{\varphi})^2\sqrt{1+\bar{\varphi}} }e^{-r_i-r_j}.
    	\end{aligned}
    \end{equation}
    Multiplying (\ref{Equation_partial_partial_upper}) and (\ref{Equation:compact_proof_cosh_upper_bound}) gives
    \begin{equation}
    	 \label{Equation:compact_proof_theta_cosh_upper_bound}
    	\begin{aligned}
    		\frac{\partial{\theta_{j}^{ik}}}{\partial{u_i}}(\cosh{l_{ij}}-1) &\leq \frac{5}{128 a^{14}(1+\bar{\varphi})^2\sqrt{1+\bar{\varphi}} }e^{-r_i-r_j} m_2e^{r_i + r_j}
    		\\[1mm]
    		& =\frac{5}{64 a^{14}(1+\bar{\varphi})^2\sqrt{1+\bar{\varphi}} }.
    	\end{aligned}    	
    \end{equation}

    Combining the lower bound estimate (\ref{Equation:compact_proof_numerator_lower_bound})
    of the numerator and the upper bound estimate (\ref{Equation:compact_proof_denominator_upper_bound}) of the denominator, we obtain a lower bound estimate of (\ref{Equation:compact_proof_basic_expression_total})
    \begin{equation}
    	\label{Equation_partial_partial_lower}
    	\begin{aligned}
    		\frac{\partial{\theta_{j}^{ik}}}{\partial{u_i}}
    		&\geq \frac{m_4e^{2r_i+2r_j+r_k}}{30e^{3r_i+3r_j+r_k}} \\[1mm]
    		& = \frac{ 16a^{10}[(1 - \varphi_{ij}^2) +
    			(\varphi_{ik} + \varphi_{ij}\varphi_{jk}) +
    			(\varphi_{jk} + \varphi_{ij}\varphi_{ik})]}{15}e^{-r_i-r_j}.
    	\end{aligned}
    \end{equation}
    Multiplying (\ref{Equation_partial_partial_lower}) and (\ref{Equation:compact_proof_cosh_lower_bound}) gives
   \begin{equation}
    \label{Equation:compact_proof_theta_cosh_lower_bound}
   \begin{aligned}
      \frac{\partial{\theta_{j}^{ik}}}{\partial{u_i}}(\cosh{l_{ij}}-1)
      &\geq \frac{ 16a^{10}[(1 - \varphi_{ij}^2) +
      	(\varphi_{ik} + \varphi_{ij}\varphi_{jk}) +
      	(\varphi_{jk} + \varphi_{ij}\varphi_{ik})]}{15}e^{-r_i-r_j} m_1 e^{r_i+r_j} \\[1mm]
      &= \frac{64a^{14}[(1 - \varphi_{ij}^2) +
                                (\varphi_{ik} + \varphi_{ij}\varphi_{jk}) +
                                (\varphi_{jk} + \varphi_{ij}\varphi_{ik})](1+\bar{\varphi})}{15}
   \end{aligned}
   \end{equation}
   By Proposition \ref{proof_middle_prop} we obtain the following lower bound
    \begin{equation}
        \label{Equation:compact_proof_theta_cosh_plus_theta_partial_lower_bound}
        \begin{split}
                &\frac{\partial{\theta_{j}^{ik}}}{\partial{u_i}}(\cosh{l_{ij}}-1)+ \frac{\partial{\theta_{k}^{ij}}}{\partial{u_i}}(\cosh{l_{ik}}-1) \\[1mm]
                &\geq
                    \frac{64a^{14}(1+\bar{\varphi})[
                        (2-\varphi_{ij}^2-\varphi_{ik}^2)+	
                        (\varphi_{ik} + \varphi_{ij}\varphi_{jk})
                    +   (\varphi_{jk} + \varphi_{ij}\varphi_{ik})
                    +	(\varphi_{ij} + \varphi_{ik}\varphi_{jk})
                    +   (\varphi_{jk} + \varphi_{ij}\varphi_{ik})
                    ]}{15} \\[1mm]
            & \geq \frac{64a^{14}(1+\bar{\varphi})(1+\bar\varphi)}{15}\\[1mm]
            & = \frac{64a^{14}(1+\bar{\varphi})^2}{15}.
        \end{split}
    \end{equation}
    By the expression (\ref{AB-equality}) of $B_{ij}$, we have $$B_{ij}(\cosh{l_{ij}}-1) + B_{ik}(\cosh{l_{ik}}-1) = (\frac{\partial{\theta_i^{jk}}}{\partial{u_j}}
+ \frac{\partial{\theta_i^{jl}}}{\partial{u_j}})(\cosh{l_{ij}}-1)+(\frac{\partial{\theta_i^{kj}}}{\partial{u_k}}
+ \frac{\partial{\theta_i^{km}}}{\partial{u_k}})(\cosh{l_{ik}}-1).$$
 Therefore, combining (\ref{Equation:compact_proof_theta_cosh_upper_bound}) and (\ref{Equation:compact_proof_theta_cosh_plus_theta_partial_lower_bound}) gives
\begin{equation*}
    \frac{64a^{14}(1+\bar{\varphi})^2}{15} \leq B_{ij}(\cosh{l_{ij}}-1) + B_{ik}(\cosh{l_{ik}}-1) \leq \frac{5}{16 a^{14}(1+\bar{\varphi})^2\sqrt{1+\bar{\varphi}} }.
\end{equation*}
Similarly, we obtain the following estimate of $A_i$
\begin{equation}
        \label{Equation:compact_proof_A_i_bound}
        \frac{64a^{14}(1+\bar{\varphi})^2}{15} \leq A_i = \sum_{j \thicksim i}{B_{ij}(\cosh{l_{ij}} - 1)} \leq \frac{5N}{64 a^{14}(1+\bar{\varphi})^2\sqrt{1+\bar{\varphi}} }.
\end{equation}
and the following estimate of $B_{ij}$
    \begin{equation}
        \label{Equation:compact_proof_B_ij_bound}
           0 \leq
            B_{ij} = {(\frac{\partial{\theta_i^{jk}}}{\partial{u_j}}
            	+ \frac{\partial{\theta_i^{jl}}}{\partial{u_j}})} \leq
                \frac{5}{ 64 a^{14}(1+\bar{\varphi})^2\sqrt{1+\bar{\varphi}} }e^{-r_i-r_j} \leq
                \frac{5}{ 64 a^{14}(1+\bar{\varphi})^2\sqrt{1+\bar{\varphi}} }.
    \end{equation}
    By (\ref{Equation:compact_proof_A_i_bound}) and (\ref{Equation:compact_proof_B_ij_bound}), set the constants as
    \begin{equation*}
    a_1(\Phi, R) = \frac{64a^{14}(1+\bar{\varphi})^2}{15}, a_2(\Phi, R) = \frac{5N}{64 a^{14}(1+\bar{\varphi})^2\sqrt{1+\bar{\varphi}}}, a_3(\Phi, R) = \frac{5}{ 64 a^{14}(1+\bar{\varphi})^2\sqrt{1+\bar{\varphi}} },
    \end{equation*}
    where $a= a(R) = \frac{1 - e^{-R}}{2}>0$ is a positive constant depending only on the constant $R>0$. We complete the proof.    $\hfill\Box$
\end{proof}

\section{The second improved estimate on the discrete Laplace operator}
In this section, we improve Theorem \ref{main result1} after removing the assumption of circle packing metrics with uniform positive lower bounds, i.e., $r_i \geq R > 0, \forall v_i \in V$ for a constant $R>0$ and prove Theorem \ref{main result2} by using the method of proof by cases. In fact, we show that for any the circle packing metric $r$, the positive upper bound constants $a_2(\Phi, R), a_3(\Phi, R)$ in Theorem \ref{main result1} do not depend on the constant $R>0$ in fact.
\begin{theorem}
    \label{Theorem2_uniform_upper_bound}
    \emph{Given a closed triangulated surface $(X, T, \Phi)$ with  $\Phi : E\rightarrow [0, \pi)$ satisfying the ($\star$) condition.
    Then for any the circle packing metric $r$, there exists positive constants $C_1(\Phi), C_2(\Phi)$ depending only on the weight $\Phi$
     such that $0<A_i \leq C_1$ and
     $0<B_{ij} \leq C_2$ for any vertex $v_i, v_j \in V$.   }
\end{theorem}
\begin{proof}
    We first prove the lower bound part of the estimates.
    Recall the definitions of $A_i$ and $B_{ij}$
    \begin{equation}
    \label{Proof_definition_A_B}
        \begin{aligned}
            B_{i j} &= \frac{\partial \theta_{i}^{j k}}{\partial u_{j}}+\frac{\partial \theta_{i}^{j l}}{\partial u_{j}}, \\
            A_{i} &= \sum_{v_j \sim v_i}B_{ij}\left(\cosh l_{ij}-1\right) = \sum_{v_j \sim v_i} \left( \frac{\partial \theta_{i}^{j k}}{\partial u_{j}}+\frac{\partial \theta_{i}^{j l}}{\partial u_{j}} \right) (\cosh l_{ij}-1),
        \end{aligned}
    \end{equation}
    where $\triangle_{ijk}$ and $\triangle_{ijl}$ denote the two triangles sharing the edge $e_{ij}$.
    The key expression is
    \begin{equation}
    \label{Equation:uniform_upper_bound_proof_theta_total_repeat_en}
    \frac{\partial{\theta_j^{ik}}}{\partial{u_i}}
    =\frac{C_kS_i^2S_j^2(1-\varphi_{ij}^2) + C_iS_iS_j^2S_k(\varphi_{ik} + \varphi_{ij}\varphi_{jk}) + C_jS_i^2S_jS_k(\varphi_{jk} + \varphi_{ij}\varphi_{ik})}
    {\sinh^2 l_{ij} \sqrt{\Delta}},
    \end{equation}
    where  $C_i = \cosh r_i, S_i = \sinh r_i, A_{ijk} = \sinh{l_{ij}}\sinh{l_{ik}}\sin{\theta_i^{jk}}, \gamma_{ijk} = \cos{\Phi_{jk}} + \cos{\Phi_{ij}}\cos{\Phi_{ik}}$ and
    \begin{equation}
    \label{Equation:uniform_upper_bound_proof_theta_delta_repeat_en}
    \begin{aligned}
    \Delta
    &= (2+2\varphi_{ij}\varphi_{jk}\varphi_{ik})S_i^2S_j^2S_k^2 + (1-\varphi_{ij}^2)S_i^2S_j^2 + (1-\varphi_{jk}^2)S_j^2S_k^2
    (1-\varphi_{ik}^2)S_i^2S_k^2 \\
    &\quad + (2\varphi_{ij} + 2\varphi_{jk}\varphi_{ik})C_iC_jS_iS_jS_k^2 + (2\varphi_{jk} + 2\varphi_{ij}\varphi_{ik})C_jC_kS_jS_kS_i^2
    (2\varphi_{ik} + 2\varphi_{ij}\varphi_{jk})C_iC_kS_iS_kS_j^2.
    \end{aligned}
    \end{equation}
    Therefore, applying the condition $\langle \star \rangle$ and the fact that $l_{ij}, l_{ik}, l_{jk} > 0$ and $\theta_{i}^{jk} \in (0,\pi)$, we have
    \begin{equation}
    \label{Proof_first_geq0}
    	\frac{\partial{\theta_j^{ik}}}{\partial{u_i}} \geq 0,
    \end{equation}
    where the equality holds when $\Phi_{ij}=0$ and $\Phi_{ik}+\Phi_{jk}=\pi$.
    Thus $B_{ij} \geq 0$ and the equality holds when $\Phi_{ij}=0$, $\Phi_{ik}+\Phi_{jk}=\pi, \Phi_{il}+\Phi_{jl}=\pi$.
    By (\ref{Proof_first_geq0}), we get that $\partial{\theta_j^{ik}}/\partial{u_i} + \partial{\theta_k^{ij}/}\partial{u_i} \geq 0$. If the equality holds, then by the conditions for the equality in (\ref{Proof_first_geq0}), we have
    \begin{equation*}
    	\Phi_{ij}=0, \Phi_{ik}+\Phi_{jk}=\pi \ \text{ and } \ \Phi_{ik}=0, \Phi_{ij}+\Phi_{jk}=\pi,
    \end{equation*}
    which implies that $\Phi_{jk} = \pi$. This is a contradiction because the range of weight is $[0,\pi)$. Thus we obtain
    \begin{equation}
    \label{Proof_second_greater_2}
    	\frac{\partial{\theta_j^{ik}}}{\partial{u_i}} + 	\frac{\partial{\theta_k^{ij}}}{\partial{u_i}} > 0.
    \end{equation}
    Using (\ref{Proof_second_greater_2}) and the fact that $A_i \geq \partial{\theta_j^{ik}}/\partial{u_i} (\cosh l_{ij} - 1) + \partial{\theta_k^{ij}/}\partial{u_i} (\cosh l_{ik} - 1)$, we obtain $A_i > 0.$

    Now we prove the upper bound part of the estimates.
    By (\ref{Proof_definition_A_B}),
    it suffices to show that
    $\partial \theta_{j}^{ik}/\partial u_i(\cosh l_{ij}-1)$
    and
    $\partial \theta_{j}^{ik}/\partial u_i$
    have uniform upper bounds.
    Set
    \begin{equation}
    x = \frac{e^{2r_i} - 1}{2},\quad y = \frac{e^{2r_j} - 1}{2},\quad z = \frac{e^{2r_k} - 1}{2}.
    \end{equation}
    Note that $r_i, r_j, r_k \in (0, +\infty)$ corresponds to $x, y, z \in (0, +\infty)$.
    We also use the following notations for some polynomials in the weights
    \begin{equation}
        \label{Equation2_a_b_c_definition}
    \begin{gathered}
    a_1 = 1-\varphi_{ij}^2, \quad a_2 = 1-\varphi_{ik}^2,\quad a_3 = 1-\varphi_{jk}^2, \\
    b_1 = \varphi_{ij}+\varphi_{ik}\varphi_{jk},\quad b_2 = \varphi_{ik}+\varphi_{ij}\varphi_{jk},\quad b_3 = \varphi_{jk}+\varphi_{ij}\varphi_{ik}, \\
    c = 1 + \varphi_{ij}\varphi_{ik}\varphi_{jk}.
    \end{gathered}
    \end{equation}
    By the condition $\langle \star \rangle$, we have $a_n \ge 0$, $b_n \ge 0$ ($n=1,2,3$) and $c > 0$.
    Now we can rewrite the numerator and denominator of (\ref{Equation:uniform_upper_bound_proof_theta_total_repeat_en}) and the term $\cosh l_{ij}-1$ as polynomials in terms of the variables $x$, $y$, and $z$. We first deal with the product of the numerator and $e^{2r_i+2r_j+r_k}$,
    \begin{equation*}
        \begin{aligned}
            & [C_kS_i^2S_j^2(1-\varphi_{ij}^2) + C_iS_iS_j^2S_k(\varphi_{ik} + \varphi_{ij}\varphi_{jk}) + C_jS_i^2S_jS_k(\varphi_{jk} + \varphi_{ij}\varphi_{ik})] e^{2r_i+2r_j+r_k} \\
            &= (1-\varphi_{ij}^2)\frac{e^{2r_k} + 1}{2}(\frac{e^{2r_i} - 1}{2})^2(\frac{e^{2r_j} - 1}{2})^2 + (\varphi_{ik} + \varphi_{ij}\varphi_{jk})\frac{e^{2r_i} + 1}{2}\frac{e^{2r_i} - 1}{2}(\frac{e^{2r_j} - 1}{2})^2\frac{e^{2r_k} - 1}{2} \\
            &\quad + (\varphi_{jk} + \varphi_{ij}\varphi_{ik})\frac{e^{2r_j} + 1}{2}(\frac{e^{2r_i} - 1}{2})^2\frac{e^{2r_j} - 1}{2}\frac{e^{2r_k} - 1}{2} \\
            &= a_1 (z+1) x^2 y^2 + b_2 (x^2+x) y^2 z + b_3 (y^2+y) x^2 z \\
            &= (a_1+b_2+b_3)x^2y^2z + b_2\,xy^2z + b_3\,x^2yz + a_1x^2y^2.
            \end{aligned}
    \end{equation*}
	Define the last line as the following polynomial denote by \(f(x, y, z)\),
    \begin{equation}
    f(x,y,z) := (a_1+b_2+b_3)x^2y^2z + b_2\,xy^2z + b_3\,x^2yz + a_1x^2y^2.
    \end{equation}

    Multiplying the first part of the denominator by $e^{2r_i+2r_j}$, we obtain
    \begin{equation*}
        \begin{aligned}
            &[(C_iC_j + \varphi_{ij}S_iS_j)^2 - 1]e^{2r_i+2r_j} \\
            &= \left( \frac{e^{r_i}+e^{-r_i}}{2}\frac{e^{r_j}+e^{-r_j}}{2} + \varphi_{ij}\frac{e^{r_i}-e^{-r_i}}{2}\frac{e^{r_j}-e^{-r_j}}{2} \right)^2 e^{2r_i+2r_j} - e^{2r_i+2r_j} \\
            &= \left( \frac{e^{2r_i}+1}{2}\frac{e^{2r_j}+1}{2} + \varphi_{ij}\frac{e^{2r_i}-1}{2}\frac{e^{2r_j}-1}{2} \right)^2  - e^{2r_i+2r_j} \\
            &= ((x+1)(y+1) + \varphi_{ij} x y )^2 - (2x+1)(2y+1) \\
            &= (3+4\varphi_{ij}+\varphi_{ij}^2)\,x^2y^2 + 2(1+\varphi_{ij})(x^2y+xy^2) + x^2+y^2+ 2\varphi_{ij}\,xy\,.
            \end{aligned}
    \end{equation*}
    Then we write
    \begin{equation}
    g_1(x,y,z) := (3+4\varphi_{ij}+\varphi_{ij}^2)\,x^2y^2 + 2(1+\varphi_{ij})(x^2y+xy^2) + x^2+y^2+ 2\varphi_{ij}\,xy\,.
    \end{equation}
    Multiplying the term $(\cosh l_{ij}-1)$ by $e^{r_i+r_j}$ gives
\begin{equation*}
        \begin{aligned}
            & (\cosh l_{ij} - 1)e^{r_i+r_j} = \bigl(C_i C_j + \varphi_{ij} S_i S_j - 1\bigr)e^{r_i+r_j} \\
            & = \left( \frac{e^{r_i}+e^{-r_i}}{2}\frac{e^{r_j}+e^{-r_j}}{2} + \varphi_{ij} \frac{e^{r_i}-e^{-r_i}}{2}\frac{e^{r_j}-e^{-r_j}}{2} - 1 \right) e^{r_i+r_j} \\
            & = \left( (x+1)(y+1) + \varphi_{ij}xy \right) - \sqrt{(2x+1)(2y+1)} \\
            & = (1+\varphi_{ij})xy + x + y + 1 - \sqrt{4xy+2x+2y+1}\,.
            \end{aligned}
    \end{equation*}
    We write the polynomial $q(x,y,z)$ defined by
    \begin{equation}
    q(x,y,z) := (1+\varphi_{ij})xy + x + y + 1 - \sqrt{4xy+2x+2y+1}\,.
    \end{equation}
    Finally, we consider the product of $\Delta$ and $e^{2r_i+2r_j+2r_k}$
    \begin{equation*}
        \begin{aligned}
            &\ e^{2r_i+2r_j+2r_k}\Delta
            \\
            &= e^{2r_i+2r_j+2r_k}[(2+2\varphi_{ij}\varphi_{jk}\varphi_{ik})S_i^2S_j^2S_k^2 + (1-\varphi_{ij}^2)S_i^2S_j^2 + (1-\varphi_{jk}^2)S_j^2S_k^2
            + (1-\varphi_{ik}^2)S_i^2S_k^2 \\
            &\quad + (2\varphi_{ij} + 2\varphi_{jk}\varphi_{ik})C_iC_jS_iS_jS_k^2 + (2\varphi_{jk} + 2\varphi_{ij}\varphi_{ik})C_jC_kS_jS_kS_i^2
            + (2\varphi_{ik} + 2\varphi_{ij}\varphi_{jk})C_iC_kS_iS_kS_j^2]
            \\
            & = (2+2\varphi_{ij}\varphi_{jk}\varphi_{ik}) (\frac{e^{2r_i}-1}{2})^2(\frac{e^{2r_j}-1}{2})^2(\frac{e^{2r_k}-1}{2})^2+(1-\varphi_{ij}^2)(\frac{e^{2r_i}-1}{2})^2(\frac{e^{2r_j}-1}{2})^2  e^{2r_k}
            \\
            & \quad + (1-\varphi_{jk}^2)(\frac{e^{2r_j}-1}{2})^2(\frac{e^{2r_k}-1}{2})^2 e^{2r_i} +(1-\varphi_{ik}^2)(\frac{e^{2r_i}-1}{2})^2(\frac{e^{2r_k}-1}{2})^2 e^{2r_j}
            \\
            & \quad + (2\varphi_{ij} + 2\varphi_{jk}\varphi_{ik})(\frac{e^{4r_i}-1}{4})(\frac{e^{4r_j}-1}{4})(\frac{e^{2r_k}-1}{2})^2 + (2\varphi_{jk} + 2\varphi_{ij}\varphi_{ik})(\frac{e^{4r_j}-1}{4})(\frac{e^{4r_k}-1}{4})(\frac{e^{2r_i}-1}{2})^2
            \\
            & \quad + (2\varphi_{ik} + 2\varphi_{ij}\varphi_{jk})(\frac{e^{4r_i}-1}{4})(\frac{e^{4r_k}-1}{4})(\frac{e^{2r_j}-1}{2})^2
            \\
            & = 2c x^2y^2z^2
                +a_1x^2y^2(2z+1) + a_3y^2z^2(2x+1) +a_2x^2z^2(2y+1) + 2b_1(x^2+x)(y^2+y)z^2
                \\
            & \quad + 2b_3(y^2+y)(z^2+z)x^2 + 2b_2(x^2+x)(z^2+z)y^2
            \\
            & = 2(c+b_1+b_2+b_3)x^2y^2z^2 + 2(a_1+b_2+b_3)x^2y^2z + 2(a_3+b_1+b_2)xy^2z^2 + 2(a_2+b_1+b_3)x^2yz^2
            \\
            & \quad + a_1x^2y^2 + a_3y^2z^2 + a_2x^2z^2 + 2b_2xy^2z + 2b_3x^2yz + 2b_1xyz^2,
            % = \sum_{i=0}^{2}\sum_{j=0}^{2}\sum_{k=0}^{2} c_{ijk}\, x^i y^j z^k,
        \end{aligned}
    \end{equation*}
    and set
    \begin{equation}
    \begin{aligned}
    	g_2(x,y,z) &=  2(a_1+b_2+b_3)x^2y^2z + 2(a_3+b_1+b_2)xy^2z^2 + 2(a_2+b_1+b_3)x^2yz^2
    	 \\
    	& \quad + a_1x^2y^2 + a_3y^2z^2 + a_2x^2z^2 + 2b_2xy^2z + 2b_3x^2yz + 2b_1xyz^2
    	\\
    	& \quad + 2(c+b_1+b_2+b_3)x^2y^2z^2.
    \end{aligned}    	
    \end{equation}
    By the preceding discussions, we have that $g_1(x,y,z) > 0$ and  $g_2(x,y,z) > 0$.
    Now we have
    \begin{equation}
        \label{Equation2:expression_partial_theta_partial_u}
        \frac{\partial{\theta_j^{ik}}}{\partial{u_i}} = \frac{f(x,y,z)e^{r_i+r_j}}{g_1(x,y,z)\sqrt{g_2(x,y,z)}} = \frac{f(x,y,z)\sqrt{(2x+1)(2y+1)}}{g_1(x,y,z)\sqrt{g_2(x,y,z)}},
    \end{equation}
    and
    \begin{equation}
        \label{Equation2:expression_partial_theta_partial_u_coshl}
        \frac{\partial \theta_j^{ik}}{\partial u_i}(\cosh l_{ij}-1) = \frac{f(x,y,z) q(x,y,z)}{g_1(x,y,z)\sqrt{g_2(x,y,z)}}.
    \end{equation}

    For convenience, set
    \begin{equation}
        T_1(x,y,z) = \frac{\partial \theta_j^{ik}}{\partial u_i}
        \quad \text{and} \quad
        T_2(x,y,z) = \frac{\partial \theta_j^{ik}}{\partial u_i} (\cosh l_{ij} -1).
    \end{equation}
    Since the numerators and denominators are all elementary functions, $T_1(x,y,z)$ and $T_2(x,y,z)$ are well-defined and continuous on the open set $(0, +\infty)^3$. Let $\Omega_1$ and $\Omega_2$ be two disjoint subsets of the domain $(0, +\infty)^3$ such that they form a partition, i.e., $(0, +\infty)^3 = \Omega_1 \cup \Omega_2$. We define these subsets as follows
    \begin{equation}
    \label{theorem_proof_three_regions}
        \begin{gathered}
            \Omega_1 = \left\{ (x, y, z) \in (0, +\infty)^3 \mid \min(x,y,z) \geq \delta \right\},
            \Omega_2 = \left\{ (x, y, z) \in (0, +\infty)^3 \mid \min(x, y, z) < \delta \right\},
        \end{gathered}
    \end{equation}
    where $\delta > 0$ is a sufficiently small constant whose value is determined by the weights $\Phi_{ij}, \Phi_{ik}, \Phi_{jk}$.

    For $\Omega_1$, using the arguments from the proof of Theorem \ref{Proposition:A_i_and_B_ij_bound_compactness}, we obtain the following bounds
    \begin{equation*}
    	T_1(x,y,z) \leq \frac{5}{128 a^{14}(1+\bar{\varphi})^2\sqrt{1+\bar{\varphi}} }, \  T_2(x,y,z) \leq \frac{5}{64 a^{14}(1+\bar{\varphi})^2\sqrt{1+\bar{\varphi}} },
    \end{equation*}
    where $a = \frac{1}{2}-\frac{1}{2\sqrt{2\delta+1}}>0$ and  $\bar{\varphi} = \min\{\,0, \min\limits_{e_{ij}\in E} \varphi_{ij}\}$.

    For $\Omega_2$, then at least one of the variables $x$, $y$, or $z$ may approach $0$, which can be divided into three cases based on the asymptotic behavior of the variables: (1) all variables approach $0$; (2) two variables approach $0$; (3) one variable approaches $0$.

    \textbf{Case 1: All three variables approach zero}

    Set $r = \sqrt{x^2 + y^2 + z^2}$ and $(\hat{x}, \hat{y}, \hat{z})=(\frac{x}{r}, \frac{y}{r}, \frac{z}{r})$ which satisfies $\hat{x}^{2}+\hat{y}^{2}+\hat{z}^{2}=1$. Since $(x,y,z) \in \mathbb{R}_{>0}^3$, we have $(\hat{x}, \hat{y}, \hat{z}) \in \mathbb{R}_{>0}^3\cap\mathbb{S}^2$. We now express $f, g_1,$ and $g_2$ as polynomials in $r$
    \begin{equation}
            \begin{aligned}
                f(x, y, z) & =  \sum_{k=4}^{5} f_k(\hat{x}, \hat{y}, \hat{z})\, r^k, \
                g_1(x, y, z) = \sum_{k=2}^{4} g_{1,k}(\hat{x}, \hat{y}, \hat{z})\, r^k, \
                g_2(x, y, z)  = \sum_{k=4}^{6} g_{2,k}(\hat{x}, \hat{y}, \hat{z})\, r^k,
            \end{aligned}
        \end{equation}
        where
        \begin{equation}
            \begin{aligned}
                f_4(\hat{x}, \hat{y}, \hat{z}) &= b_2 \hat{x} \hat{y}^2 \hat{z} + b_3 \hat{x}^2 \hat{y} \hat{z} + a_1 \hat{x}^2 \hat{y}^2, \\
                f_5(\hat{x}, \hat{y}, \hat{z}) &= (a_1+b_2+b_3)\hat{x}^2\hat{y}^2\hat{z}, \\
            \end{aligned}
        \end{equation}
        and
        \begin{equation}
            \begin{aligned}
                g_{1,2}(\hat{x}, \hat{y}, \hat{z}) &= \hat{x}^2 + \hat{y}^2 + 2\varphi_{ij} \hat{x} \hat{y}, \\
                g_{1,3}(\hat{x}, \hat{y}, \hat{z}) &= 2(1 + \varphi_{ij})(\hat{x}^2 \hat{y} + \hat{x} \hat{y}^2), \\
                g_{1,4}(\hat{x}, \hat{y}, \hat{z}) &= (3 + 4\varphi_{ij} + \varphi_{ij}^2)\hat{x}^2 \hat{y}^2, \\
            \end{aligned}
        \end{equation}
        and
        \begin{equation}
            \begin{aligned}
                g_{2,4}(\hat{x}, \hat{y}, \hat{z}) &= a_1 \hat{x}^2 \hat{y}^2 + a_3 \hat{y}^2 \hat{z}^2 + a_2 \hat{x}^2 \hat{z}^2 + 2b_2 \hat{x} \hat{y}^2 \hat{z} + 2b_3 \hat{x}^2 \hat{y} \hat{z} + 2b_1 \hat{x} \hat{y} \hat{z}^2, \\
                g_{2,5}(\hat{x}, \hat{y}, \hat{z}) &= 2(a_1 + b_2 + b_3)\hat{x}^2 \hat{y}^2 \hat{z} + 2(a_3 + b_1 + b_2)\hat{x} \hat{y}^2 \hat{z}^2 + 2(a_2 + b_1 + b_3)\hat{x}^2 \hat{y} \hat{z}^2, \\
                g_{2,6}(\hat{x}, \hat{y}, \hat{z}) &= 2(c+b_1+b_2+b_3)\hat{x}^2 \hat{y}^2 \hat{z}^2, \\
            \end{aligned}
        \end{equation}
	If $a_1 = b_2 = b_3 = 0$, we have $f \equiv 0$. Hence there exists a constant $\delta > 0$ such that $(x,y,z) \in B((0,0,0), \delta)$, $T_1(x,y,z) = 0, T_2(x,y,z) = 0$.

    If $a_1 + b_2 + b_3 \neq 0$, dividing both the numerator and the denominator of $T_1(x,y,z)$ by $r^4$ gives
        \begin{equation*}
            \begin{aligned}
                T_1(x,y,z) &= \frac{f(x, y, z) \sqrt{(1+2x)(1+2y)}}{g_1(x, y, z)\sqrt{g_2(x, y, z)}}
                \\
                & = \frac{(f_4(\hat{x}, \hat{y}, \hat{z})+ f_5(\hat{x}, \hat{y}, \hat{z}) r) \sqrt{(1+2\hat{x}r)(1+2\hat{y}r)}}{(g_{1,2}(\hat{x}, \hat{y}, \hat{z}) + \sum_{k=3}^4 g_{1,k}(\hat{x}, \hat{y}, \hat{z})\, r^{k-2})\sqrt{g_{2,4}(\hat{x}, \hat{y}, \hat{z}) + g_{2,5}(\hat{x}, \hat{y}, \hat{z}) r}}.
            \end{aligned}
        \end{equation*}
    Hence we have
    \begin{equation*}
        \begin{aligned}
            T_1(x,y,z) &\leq \frac{(f_4(\hat{x}, \hat{y}, \hat{z})+ f_5(\hat{x}, \hat{y}, \hat{z}) r) \sqrt{(1+2\hat{x}r)(1+2\hat{y}r)}}{(g_{1,2}(\hat{x}, \hat{y}, \hat{z}) )\sqrt{g_{2,4}(\hat{x}, \hat{y}, \hat{z})}}
            \\
            &\leq \frac{f_4(\hat{x}, \hat{y}, \hat{z})(1+r)\sqrt{(1+2\hat{x}r)(1+2\hat{y}r)}}{(g_{1,2}(\hat{x}, \hat{y}, \hat{z}) )\sqrt{g_{2,4}(\hat{x}, \hat{y}, \hat{z})}}.
        \end{aligned}
    \end{equation*}
    In the first line, we used $g_{1,3}, g_{1,4}, g_{2,5} \geq 0$,
    and in the second line,
    the fact that
    $\hat{x}^2 \hat{y}^2 \hat{z} \leq \hat{x} \hat{y}^2 \hat{z}$,
    $\hat{x}^2 \hat{y}^2 \hat{z} \leq \hat{x}^2  \hat{y} \hat{z} $
    and $\hat{x}^2 \hat{y}^2 \hat{z} \leq \hat{x}^2  \hat{y}^2 $
    which are consequences of  $\hat{x}, \hat{y}, \hat{z} \in (0,1)$.
    Hence for a samll enough $\epsilon$ there exists a small enough constant $\delta_0(\epsilon)$ such that $r\leq \delta_0$, it follows that
    \begin{equation}
    	T_1(x,y,z) \leq \frac{f_4(\hat{x}, \hat{y}, \hat{z})}{g_{1,2}(\hat{x}, \hat{y}, \hat{z}) \sqrt{g_{2,4}(\hat{x}, \hat{y}, \hat{z})}} + \epsilon.
    \end{equation}
    It remains to prove  $\frac{f_4(\hat{x}, \hat{y}, \hat{z})}{g_{1,2}(\hat{x}, \hat{y}, \hat{z}) \sqrt{g_{2,4}(\hat{x}, \hat{y}, \hat{z})}}$ is bounded. We now employ the notation
        \begin{equation}
            \begin{aligned}
                s(\hat{x}, \hat{y}, \hat{z}) &= \frac{f_4(\hat{x}, \hat{y}, \hat{z})}{g_{1,2}(\hat{x}, \hat{y}, \hat{z}) \sqrt{g_{2,4}(\hat{x}, \hat{y}, \hat{z})}}
                \\
                &= \frac{b_2 \hat{x} \hat{y}^2 \hat{z} + b_3 \hat{x}^2 \hat{y} \hat{z} + a_1 \hat{x}^2 \hat{y}^2}{(\hat{x}^2 + \hat{y}^2 + 2\varphi_{ij} \hat{x} \hat{y}) \sqrt{a_1 \hat{x}^2 \hat{y}^2 + a_3 \hat{y}^2 \hat{z}^2 + a_2 \hat{x}^2 \hat{z}^2 + 2b_2 \hat{x} \hat{y}^2 \hat{z} + 2b_3 \hat{x}^2 \hat{y} \hat{z} + 2b_1 \hat{x} \hat{y} \hat{z}^2}}.
            \end{aligned}
        \end{equation}
    If $a_3 \neq 0, a_2 \neq 0$, we obtain
    \begin{equation*}
    \begin{aligned}
    s(\hat{x}, \hat{y}, \hat{z}) &\leq \frac{b_2  \hat{x} \hat{y}^2 \hat{z}}{(2+2\varphi_{ij}) xy \sqrt{a_3 y^2 z^2 }} + \frac{b_3 \hat{x}^2 \hat{y} \hat{z}}{(2+2\varphi_{ij}) xy \sqrt{a_2 \hat{x}^2 \hat{z}^2 }} + \frac{a_1 \hat{x}^2 \hat{y}^2}{(2+2\varphi_{ij})xy \sqrt{a_1 \hat{x}^2 \hat{y}^2}}
        	\\
    &\leq \frac{b_2}{(2+2\varphi_{ij})\sqrt{a_3}} + \frac{b_3}{(2+2\varphi_{ij})\sqrt{a_2}} + \frac{a_1}{(2+2\varphi_{ij})\sqrt{a_1}}.
    \end{aligned}  	
   	\end{equation*}
   	If $a_3 = 0, a_2 \neq 0$, we have $b_3 = \varphi_{jk} + \varphi_{ij} \varphi_{ik} = 1 + \varphi_{ij} \varphi_{ik} > 0$ and $b_2 = b_1 = \varphi_{ij} + \varphi_{ik}$. Hence we obtain
   	\begin{equation*}
   	\begin{aligned}
   	   		s(\hat{x}, \hat{y}, \hat{z}) &\leq \frac{b_2 \hat{x} \hat{y}^2 \hat{z} + b_3 \hat{x}^2 \hat{y} \hat{z} + a_1 \hat{x}^2 \hat{y}^2}{(\hat{x}^2 + \hat{y}^2 + 2\varphi_{ij} \hat{x} \hat{y}) \sqrt{a_1 \hat{x}^2 \hat{y}^2 + a_2 \hat{x}^2 \hat{z}^2 + 2b_3 \hat{x}^2 \hat{y} \hat{z} }}
   	   		\\
   	   		& = \frac{b_2  \hat{y}^2 \hat{z} + b_3 \hat{x} \hat{y} \hat{z} + a_1 \hat{x} \hat{y}^2}{(\hat{x}^2 + \hat{y}^2 + 2\varphi_{ij} \hat{x} \hat{y}) \sqrt{a_1 \hat{y}^2 + a_2  \hat{z}^2 + 2b_3  \hat{y} \hat{z} }}.
   	\end{aligned}
   	\end{equation*}
       The denominator of the expression vanishes if either $(\hat{x}, \hat{y}, \hat{z}) \to (0,0,1)$ or $(\hat{x}, \hat{y}, \hat{z}) \to (1,0,0)$. Consider the first case. As $(\hat{x}, \hat{y}, \hat{z}) \to (0,0,1)$ along the specific path defined by $\hat{y} = k\hat{x}$, where $k \in (0,+\infty)$ is an arbitrary constant, we have
   	\begin{equation*}
   	\label{equation:proof_s_hatx_haty_hatz_1}
   	\begin{aligned}
   	   		s(\hat{x}, \hat{y}, \hat{z}) &\leq  \frac{b_2  k^2 \hat{x}^2 \hat{z} + b_3 k\hat{x}^2  \hat{z} + a_1 k^2 \hat{x}^3}{(\hat{x}^2 + k^2\hat{x}^2 + 2\varphi_{ij}k \hat{x}^2) \sqrt{a_1 k^2\hat{x}^2 + a_2  \hat{z}^2 + 2b_3  k\hat{x} \hat{z} }}
   	   		\\
   	   		&\leq \frac{b_2  k^2 \hat{z}  + b_3 k \hat{z} + a_1 k^2 \hat{x}}{(1 + k^2 + 2\varphi_{ij}k ) \sqrt{ a_2  \hat{z}^2 }} .
   	\end{aligned}
   	\end{equation*}
   	Since $\hat{x} \to 0, \hat{z} \to 1$, we obtain that if $k \to 0$, the right term tends to $0$, which implies $s(\hat{x}, \hat{y}, \hat{z}) \leq \epsilon$ for a samll enough constant $\epsilon$.
    If $k \to +\infty$, $s(\hat{x}, \hat{y}, \hat{z}) \leq \frac{b_2}{\sqrt{a_2}}$.
    Hence, we obtain that for any $k$, there exists a constant $c_0>0$ such that $s(\hat{x}, \hat{y}, \hat{z}) \leq c_0$ along the fixed path $\hat{y} = k \hat{x}$ to $(0,0,1)$, which implies that along any path to $(0,0,1)$ we have $s(\hat{x}, \hat{y}, \hat{z}) \leq c_{0}$. Hence there exists $\delta_0> 0$  such that if $(\hat{x}, \hat{y}, \hat{z}) \in B((0,0,1), \delta_0)$, we have $s(\hat{x}, \hat{y}, \hat{z}) \leq c_0$.
   	
   	For $(\hat{x}, \hat{y}, \hat{z}) \to (1,0,0)$, using analogous methods we show that there exists a constant $c_1>0$ and a correspending constant $\delta_1 > 0$ such that if $(\hat{x}, \hat{y}, \hat{z}) \in B((0,0,1), \delta_1)$, we have  $s(\hat{x}, \hat{y}, \hat{z}) \leq c_1$.
   	
   	For the compact set $(\hat{x}, \hat{y}, \hat{z}) \in \{\hat{x}^2 + \hat{y}^2 + \hat{z}^2 = 1\} / ( B((0,0,1), \delta_0) \cup B((1,0,0), \delta_1))$, since the denominator is not zero , using the continuity of the numerator and the denominator, we obtain that there exists a constant $c_2>0$ such that $s(\hat{x}, \hat{y}, \hat{z}) \leq c_2$.
   	
   	Set $C = \max\{c_0, c_1, c_2\}>0$, and we have
   	\begin{equation*}
   		s(\hat{x}, \hat{y}, \hat{z}) \leq C, \text{ if } a_3 = 0, a_2 \neq 0.
   	\end{equation*}
   	
   	The case where $a_2 = 0, a_3 \neq 0$ is analogous. 	
    We now consider $a_2 = a_3 =0$, which  implies that $b_2 > 0, b_3 > 0$.  Since $b_1 = b_3 = b_2$ in this case, we have  $b_1 > 0$. Hence we obtain
    \begin{equation*}
    \begin{aligned}
        	s(\hat{x}, \hat{y}, \hat{z}) &\leq \frac{b_2 \hat{x} \hat{y}^2 \hat{z} + b_3 \hat{x}^2 \hat{y} \hat{z} + a_1 \hat{x}^2 \hat{y}^2}{(\hat{x}^2 + \hat{y}^2 + 2\varphi_{ij} \hat{x} \hat{y}) \sqrt{a_1 \hat{x}^2 \hat{y}^2 + 2b_2 \hat{x} \hat{y}^2 \hat{z} + 2b_3 \hat{x}^2 \hat{y} \hat{z} + 2b_1 \hat{x} \hat{y} \hat{z}^2}} \\
        	&\leq \frac{b_2 \hat{x}^{\frac{1}{2}} \hat{y}^\frac{3}{2} \hat{z} + b_3 \hat{x}^\frac{3}{2} \hat{y}^\frac{1}{2} \hat{z} + a_1 \hat{x}^\frac{3}{2} \hat{y}^\frac{3}{2}}{(\hat{x}^2 + \hat{y}^2 + 2\varphi_{ij} \hat{x} \hat{y}) \sqrt{a_1 \hat{x} \hat{y} + 2b_2  \hat{y} \hat{z} + 2b_3 \hat{x}  \hat{z} + 2b_1 \hat{z}^2}}.
    \end{aligned}
    \end{equation*}
    The denominator tends to zero only if $(\hat{x}, \hat{y}, \hat{z}) \to (0,0,1)$. Along the specific path $\hat{y} = k\hat{x}$, for any constant $k \in (0,+\infty)$, we obtain
    \begin{equation*}
    	s(\hat{x}, \hat{y}, \hat{z}) \leq \frac{b_2 k^{\frac{3}{2}}  \hat{z} + b_3 k^{\frac{1}{2}}  \hat{z} + a_1 k^{\frac{3}{2}} \hat{x} }{(1+ k^2 + 2\varphi_{ij} k)  \sqrt{a_1 \hat{x} \hat{y} + 2b_2  \hat{y} \hat{z} + 2b_3 \hat{x}  \hat{z} + 2b_1 \hat{z}^2}}.
    \end{equation*}
    Since $(\hat{x}, \hat{y}, \hat{z}) \to (0,0,1)$, we obtain that if $k \to 0$, the right term of the above equality tends to $0$, which implies that $s(\hat{x}, \hat{y}, \hat{z}) \leq 1$.
    If $k \to +\infty$, the right term of the above equality tends to $0$ which implies that $s(\hat{x}, \hat{y}, \hat{z}) \leq 1$.
    Hence, using analogous methods we obtain that there exists $\delta_0> 0$  such that if $(\hat{x}, \hat{y}, \hat{z}) \in B((0,0,1), \delta_0)$, we have $s(\hat{x}, \hat{y}, \hat{z}) \leq c_3$.

    For the compact set $(\hat{x}, \hat{y}, \hat{z}) \in \mathbb{R}_{\geq 0}^3 \cap \{\hat{x}^2 + \hat{y}^2 + \hat{z}^2 = 1\} / ( B((0,0,1), \delta_0))$, since the denominator is not zero, using the continuity of the numerator and the denominator, we obtain that there exists a constant $c_4$ such that $s(\hat{x}, \hat{y}, \hat{z}) \leq c_4$.

    Let $C = \max\{c_3, c_4\}$, we have
    \begin{equation*}
        s(\hat{x}, \hat{y}, \hat{z}) \leq C, \text{ if } a_2 = a_3 = 0.
    \end{equation*}

    Summarizing the preceding discussion of the various cases, we obtain the following
    \begin{equation}
        \label{equation:proof_delta_0_estimate}
        \exists C > 0, \exists \delta_0 > 0, \forall (x,y,z) \in B((0,0,0), \delta_0), T_1(x,y,z) \leq C.
    \end{equation}
    Since $T_2(x,y,z) = T_1(x,y,z) (\cosh l_{ij} -1)$, using the relation of $x,y,z$ of $r_i, r_j, r_k$, we get that if $(x,y,z) \in B((0,0,0), \delta_0)$ for some small $\delta_0$, then $\cosh l_{ij} -1 = C_i C_j + \varphi_{ij} S_i S_j - 1 \leq 1$. Hence we also have $T_2(x,y,z) \leq C$.

    \textbf{Case 2: Two variables approach zero}

         When exactly two variables vanish, three distinct sub-cases arise, corresponding to whether \(x = 0\), \(y = 0\), or \(z = 0\).
         First consider the sub-case where \((x,y,z) \in \mathbb{R}^3_{>0}\) approaches
          \((0, 0, \bar{z}) \in D\), with any \(\bar{z}\geq \delta_0\) where $\delta_0$ is choosed in (\ref{equation:proof_delta_0_estimate}).
        \begin{itemize}
            \item $ x \to 0, y \to 0, z \to \bar{z}\geq \delta_0 $
        \end{itemize}
        Let $\rho_{xy} = \sqrt{x^2 + y^2}$ and $x = \rho_{xy} \hat{x} , y= \rho_{xy} \hat{y}$. Since the region is in $\mathbb{R}_{>0}^3$, we have $0 < \hat{x} <1$ and $0 < \hat{y} <1$. We express both the numerator and the denominator as polynomials in the variable $\rho_{xy}$
        \begin{equation}
        \begin{aligned}
            f(x, y, z) &= \sum_{k=3}^{4} f_k(\hat{x}, \hat{y}, z)\, \rho_{xy}^k, \quad
            g_1(x, y, z) = \sum_{k=2}^{4} g_{1,k}(\hat{x}, \hat{y})\, \rho_{xy}^k, \quad
            g_2(x, y, z) = \sum_{k=2}^{4} g_{2,k}(\hat{x}, \hat{y}, z)\, \rho_{xy}^k,
        \end{aligned}
        \end{equation}
        where
        \begin{equation}
            \begin{aligned}
                f_3(\hat{x}, \hat{y}, z) &= b_2 \hat{x} \hat{y}^2 z + b_3 \hat{x}^2 \hat{y} z, \\
                f_4(\hat{x}, \hat{y}, z) &= a_1 \hat{x}^2 \hat{y}^2 + (a_1 + b_2 + b_3) \hat{x}^2 \hat{y}^2 z,
        	 \end{aligned}
        \end{equation}
        and
        \begin{equation}
            \begin{aligned}
                g_{1,2}(\hat{x}, \hat{y}) &= \hat{x}^2 + \hat{y}^2 + 2\varphi_{ij} \hat{x} \hat{y} = 1 + 2\varphi_{ij} \hat{x} \hat{y}, \\
                g_{1,3}(\hat{x}, \hat{y}) &= 2(1 + \varphi_{ij})(\hat{x}^2 \hat{y} + \hat{x} \hat{y}^2), \\
                g_{1,4}(\hat{x}, \hat{y}) &= (3 + 4\varphi_{ij} + \varphi_{ij}^2)\hat{x}^2 \hat{y}^2,
            \end{aligned}
        \end{equation}
        and
        \begin{equation}
                \begin{aligned}
                g_{2,2}(\hat{x}, \hat{y}, z) &= a_3 \hat{y}^2 z^2 + a_2 \hat{x}^2 z^2 + 2b_1 \hat{x} \hat{y} z^2, \\
                g_{2,3}(\hat{x}, \hat{y}, z) &= 2(a_3 + b_1 + b_2)\hat{x} \hat{y}^2 z^2 + 2(a_2 + b_1 + b_3)\hat{x}^2 \hat{y} z^2 + 2b_2 \hat{x} \hat{y}^2 z + 2b_3 \hat{x}^2 \hat{y} z, \\
                g_{2,4}(\hat{x}, \hat{y}, z) &= a_1 \hat{x}^2 \hat{y}^2 + 2(a_1 + b_2 + b_3)\hat{x}^2 \hat{y}^2 z + 2(c + b_1 + b_2 + b_3)\hat{x}^2 \hat{y}^2 z^2.
            \end{aligned}
        \end{equation}
        The condition $\langle \star \rangle$ implies $a_3 + a_2 + b_1 = 1 - \varphi_{jk}^2 + 1 - \varphi_{ik}^2 + \varphi_{ij} + \varphi_{ik}\varphi_{jk} \geq 0$. The equality holds if and only if $\varphi_{jk} = 1, \varphi_{ik} = 1, \varphi_{ij} + \varphi_{ik} \varphi_{jk} = 0$. However,this is impossible considering the range of $\varphi$. Consequently, $a_3 + a_2 + b_1 > 0$, and as a result, $g_{2,2} > 0$. Noting that $\hat{x}^2 + \hat{y}^2 = 1$ and $-1 < \varphi_{ij} \leq 1$, it is clear that $1+2\varphi_{ij} \hat{x} \hat{y} > 0$, which implies that $g_{1,2} > 0$.

        If $b_2 = b_3 = 0$, then $f_3 \equiv 0$.  By dividing both the numerator and denominator of $T_1(x,y,z)$ by the term $\rho_{xy}^{3}$, the expression becomes
        \begin{equation*}
        \begin{aligned}
            T_1(x,y,z) &= \frac{f_4(\hat{x}, \hat{y}, z) \rho_{xy}^{4} \sqrt{(1 + 2\rho_{xy} \hat{x})(1+2\rho_{xy} \hat{y})}}{( \sum_{k=2}^{4} g_{1,k}(\hat{x}, \hat{y})\, \rho_{xy}^k) \sqrt{ \sum_{k=2}^{4} g_{2,k}(\hat{x}, \hat{y}, z)\, \rho_{xy}^k}}
            \\
            &= \frac{f_4(\hat{x}, \hat{y}, z) \rho_{xy} \sqrt{(1 + 2\rho_{xy} \hat{x})(1+2\rho_{xy} \hat{y})}}{( \sum_{k=2}^{4} g_{1,k}(\hat{x}, \hat{y})\, \rho_{xy}^{k-2}) \sqrt{ \sum_{k=2}^{4} g_{2,k}(\hat{x}, \hat{y}, z)\, \rho_{xy}^{k-2}}}.
        \end{aligned}
        \end{equation*}
        However, $f_4(\hat{x}, \hat{y}, z)$ tends to infinity as $z \to +\infty$. Therefore, by dividing both the numerator and the denominator by $z$, we obtain
        \begin{equation*}
            \begin{aligned}
                T_1(x,y,z) &= \frac{\frac{f_4(\hat{x}, \hat{y}, z) }{z}\rho_{xy} \sqrt{(1 + 2\rho_{xy} \hat{x})(1+2\rho_{xy} \hat{y})}}{( \sum_{k=2}^{4} g_{1,k}(\hat{x}, \hat{y})\, \rho_{xy}^{k-2}) \sqrt{ \frac{\sum_{k=2}^{4} g_{2,k}(\hat{x}, \hat{y}, z)\, \rho_{xy}^{k-2}}{z^2}}}.
            \end{aligned}
        \end{equation*}
        Since $\bar{z} \geq \delta_0$, we have $0 < \frac{1}{\bar{z}} \leq \frac{1}{\delta_0}$. Considering this and the continuity of the numerator and the denominator, it follows that for any $\epsilon > 0$, there exists a constant $\delta_1 > 0$ such that if $(x,y,z) \in B((0,0,\bar{z}), \delta_1)$ for any $\bar{z} \geq \delta_0$, we have
        \begin{equation*}
            T_1(x,y,z) = \frac{\frac{f_4(\hat{x}, \hat{y}, z) }{z}\rho_{xy} \sqrt{(1 + 2\rho_{xy} \hat{x})(1+2\rho_{xy} \hat{y})}}{( \sum_{k=2}^{4} g_{1,k}(\hat{x}, \hat{y})\, \rho_{xy}^{k-2}) \sqrt{ \frac{\sum_{k=2}^{4} g_{2,k}(\hat{x}, \hat{y}, z)\, \rho_{xy}^{k-2}}{z^2}}} \leq \epsilon.
        \end{equation*}

        When $b_2 \ne 0$ or $b_3 \ne 0$, the expression $T_1(x,y,z)$ can be rewritten as follows by dividing both its numerator and denominator by $\rho_{xy}^{3}$
        \begin{equation*}
            \begin{aligned}
                T_1(x,y,z) &= \frac{(\sum_{k=3}^{4} f_k(\hat{x}, \hat{y}, z)\, \rho_{xy}^k) \sqrt{(1 + 2\rho_{xy} \hat{x})(1+2\rho_{xy} \hat{y})}}{( \sum_{k=2}^{4} g_{1,k}(\hat{x}, \hat{y})\, \rho_{xy}^k) \sqrt{ \sum_{k=2}^{4} g_{2,k}(\hat{x}, \hat{y}, z)\, \rho_{xy}^k}}
                \\
                &= \frac{(\sum_{k=3}^{4} f_k(\hat{x}, \hat{y}, z)\, \rho_{xy}^{k-3}) \sqrt{(1 + 2\rho_{xy} \hat{x})(1+2\rho_{xy} \hat{y})}}{( \sum_{k=2}^{4} g_{1,k}(\hat{x}, \hat{y})\, \rho_{xy}^{k-2}) \sqrt{ \sum_{k=2}^{4} g_{2,k}(\hat{x}, \hat{y}, z)\, \rho_{xy}^{k-2}}}.
            \end{aligned}
        \end{equation*}
        The numerator and denominator both include zeroth-order terms in $\rho_{xy}$, which implies that $T_1(x,y,z)$ may not converge.
        Analogously, to ensure the boundedness of both the numerator and denominator, we divide them by $z$, which gives
        \begin{equation*}
            T_1(x,y,z) = \frac{\frac{(\sum_{k=3}^{4} f_k(\hat{x}, \hat{y}, z)\, \rho_{xy}^{k-3})}{z} \sqrt{(1 + 2\rho_{xy} \hat{x})(1+2\rho_{xy} \hat{y})}}{( \sum_{k=2}^{4} g_{1,k}(\hat{x}, \hat{y})\, \rho_{xy}^{k-2}) \sqrt{ \frac{\sum_{k=2}^{4} g_{2,k}(\hat{x}, \hat{y}, z)\, \rho_{xy}^{k-2}}{z^2}}}.
        \end{equation*}
        Using $0 < \frac{1}{\bar{z}} \leq \frac{1}{\delta_0}$ and the continuity of the numerator and the denominator, it follows that for any $\epsilon > 0$, there exists a constant $\delta_1 > 0$ such that if $(x,y,z) \in B((0,0,\bar{z}), \delta_1)$ for any $\bar{z} \geq \delta_0$, we have
        \begin{equation*}
            T_1(x,y,z) \leq \frac{\frac{f_3(\hat{x}, \hat{y}, \bar{z})}{z}}{g_{1,2}(\hat{x}, \hat{y}) \sqrt{\frac{g_{2,2}(\hat{x}, \hat{y}, \bar{z})}{z^2}}} +  \epsilon = \frac{b_2 \hat{x} \hat{y}^2  + b_3 \hat{x}^2 \hat{y} }{(1 + 2\varphi_{ij} \hat{x} \hat{y}) \sqrt{a_3 \hat{y}^2 + a_2 \hat{x}^2 + 2b_1 \hat{x} \hat{y}}} +  \epsilon.
        \end{equation*}
        It remains to demonstrate the boundedness of the right-hand side of the above inequality. We now employ the notation
        \begin{equation*}
            s(\hat{x}, \hat{y}) = \frac{b_2 \hat{x} \hat{y}^2  + b_3 \hat{x}^2 \hat{y} }{(1 + 2\varphi_{ij} \hat{x} \hat{y}) \sqrt{a_3 \hat{y}^2 + a_2 \hat{x}^2 + 2b_1 \hat{x} \hat{y}}}.
        \end{equation*}

        If $a_3 \neq 0$ and $a_2 \neq 0$, we are concerned with $(\hat{x}, \hat{y})$ such that $\hat{x}^2 + \hat{y}^2 = 1$ and $0 < \hat{x}, \hat{y} \leq 1$. Consider now the slightly enlarged compact domain where $\hat{x}^2 + \hat{y}^2 = 1$ and $0 \leq \hat{x}, \hat{y} \leq 1$.  In this domain, the numerator $b_2 \hat{x} \hat{y}^2 + b_3 \hat{x}^2 \hat{y}$ is continuous and thus achieves a maximum value. The denominator $(1 + 2\varphi_{ij} \hat{x} \hat{y}) \sqrt{a_3 \hat{y}^2 + a_2 \hat{x}^2 + 2b_1 \hat{x} \hat{y}}$, is also continuous and attains a minimum which is strictly positive. Therefore, we can find a constant $C_1 > 0$ such that $s(\hat{x}, \hat{y}) \leq C_1$ for any unit vector $(\hat{x},\hat{y})$.

        If $a_3 = 0$ and $a_2 \neq 0$, we divide $s(\hat{x}, \hat{y})$ by $x^{\frac{1}{2}}$ as follows
        \begin{equation*}
            s(\hat{x}, \hat{y}) = \frac{b_2 \hat{x}^{\frac{1}{2}} \hat{y}^{2} + b_3 \hat{x}^\frac{3}{2} \hat{y}}{(1+\varphi_{ij}\hat{x}\hat{y}) \sqrt{a_2 \hat{x} + 2b_1 \hat{y}}}.
        \end{equation*}
        Hence we obtain
        \begin{equation*}
            \begin{aligned}
                s(\hat{x}, \hat{y}) &\leq \frac{b_2 \hat{x}^{\frac{1}{2}} \hat{y}^{2} + b_3 \hat{x}^\frac{3}{2} \hat{y}}{(1+\varphi_{ij}\hat{x}\hat{y}) \sqrt{a_2 \hat{x}}}
                \\
                & \leq \frac{b_2 \hat{y}^{2} + b_3 \hat{x} \hat{y}}{(1+\varphi_{ij}\hat{x}\hat{y}) \sqrt{a_2}}
                \\
                & \leq \frac{b_2 + b_3 }{(1+\bar{\varphi})\sqrt{a_2}}.
            \end{aligned}
        \end{equation*}
        The case where $a_3 \neq 0$ and $a_2 = 0$ is analogous. \par
        If $a_3 = 0$ and $a_2 = 0$, in this case, $b_1 = b_2 = b_3 = \varphi_{ij} + 1$. Hence we obtain
        \begin{equation*}
            \begin{aligned}
                s(\hat{x}, \hat{y}) &= \frac{b_2 \hat{x}^{\frac{1}{2}} \hat{y}^{\frac{3}{2}} + b_3 \hat{x}^{\frac{3}{2}} \hat{y}^{\frac{1}{2}}}{(1+\varphi_{ij} \hat{x} \hat{y}) \sqrt{2b_1}}
                \\
                & \leq \frac{b_2 + b_3 }{(1+\bar{\varphi}) \sqrt{2b_1}}.
            \end{aligned}
        \end{equation*}
        Summarizing the preceding discussion of the various cases, we have
    \begin{equation}
        \label{equation:proof_delta_1_estimate}
        \exists C > 0, \exists \delta_1 > 0, \forall \bar{z} \geq \delta_0, \forall (x,y,z) \in B((0,0,\bar{z}), \delta_1), T_1(x,y,z) \leq C.
    \end{equation}
    Since $T_2(x,y,z) = T_1(x,y,z) (\cosh l_{ij} -1)$, using the relation of $x,y,z$ of $r_i, r_j, r_k$, we obtain that if $(x,y,z) \in B((0,0,\bar{z}), \delta_1)$ for some small $\delta_0$ and any $\bar{z} \geq \delta_0$, then $T_2(x,y,z) \leq C$.
        \begin{itemize}
            \item $ x \to 0, z \to 0, y \to \bar{y} \geq \delta_0 $
        \end{itemize}
        Let $\rho_{xz} = \sqrt{x^2 + z^2}$ and $x = \rho_{xz} \hat{x} , z= \rho_{xz} \hat{z}$. Since the region is in $\mathbb{R}_{>0}^3\cap\mathbb{S}^2$, we have $0 < \hat{x} <1$ and $0 < \hat{z} <1$. We express both the numerator and the denominator as polynomials in the variable $\rho_{xz}$
        \begin{equation}
        \begin{aligned}
            f(x, y, z) &= \sum_{k=2}^{3} f_k(\hat{x}, y, \hat{z})\, \rho_{xz}^k, \quad
            g_1(x, y, z) = \sum_{k=0}^{2} g_{1,k}(\hat{x}, y, \hat{z})\, \rho_{xz}^k, \quad
            g_2(x, y, z) = \sum_{k=2}^{4} g_{2,k}(\hat{x}, y, \hat{z})\, \rho_{xz}^k,
        \end{aligned}
        \end{equation}
        where
        \begin{equation}
            \begin{aligned}
                f_2(\hat{x}, y, \hat{z}) &= a_1 \hat{x}^2 y^2 + b_2 \hat{x} y^2 \hat{z}, \\
                f_3(\hat{x}, y, \hat{z}) &= (a_1 + b_2 + b_3)\hat{x}^2 y^2 \hat{z} + b_3 \hat{x}^2 y \hat{z}, \\
            \end{aligned}
        \end{equation}
        and
        \begin{equation}
            \begin{aligned}
                g_{1,0}(\hat{x}, y, \hat{z}) &= y^2, \\
                g_{1,1}(\hat{x}, y, \hat{z}) &= 2(1+\varphi_{ij})\hat{x}y^2 + 2\varphi_{ij}\hat{x}y, \\
                g_{1,2}(\hat{x}, y, \hat{z}) &= (3 + 4\varphi_{ij} + \varphi_{ij}^2)\hat{x}^2 y^2 + 2(1 + \varphi_{ij})\hat{x}^2 y + \hat{x}^2, \\
            \end{aligned}
        \end{equation}
        and
        \begin{equation}
            \begin{aligned}
                g_{2,2}(\hat{x}, y, \hat{z}) &= a_1\hat{x}^2y^2 + a_3y^2\hat{z}^2 + 2b_2\hat{x}y^2\hat{z}, \\
                g_{2,3}(\hat{x}, y, \hat{z}) &= 2(a_1+b_2+b_3)\hat{x}^2y^2\hat{z} + 2(a_3+b_1+b_2)\hat{x}y^2\hat{z}^2 + 2b_3\hat{x}^2y\hat{z} + 2b_1\hat{x}y\hat{z}^2, \\
                g_{2,4}(\hat{x}, y, \hat{z}) &= 2(a_2+b_1+b_3)\hat{x}^2y\hat{z}^2 + a_2\hat{x}^2\hat{z}^2 + 2(c+b_1+b_2+b_3)\hat{x}^2y^2\hat{z}^2.
            \end{aligned}
        \end{equation}
        Since $a_1 + a_3 + b_2 = 1 - \varphi_{ij}^2 + 1 - \varphi_{jk}^2 + \varphi_{ik} + \varphi_{ij} \varphi_{jk} > 0$, we deduce $g_{2,2} > 0$.  By dividing the numerator and denominator of $T_1(x,y,z)$ by $\rho_{xz}$, we obtain
        \begin{equation*}
            \begin{aligned}
                T_1(x,y,z) = \frac{(\sum_{k=2}^{3} f_k(\hat{x}, y, \hat{z})\, \rho_{xz}^{k-1}) \sqrt{(1 + 2\rho_{xz} \hat{x})(1+2y)}}{( \sum_{k=0}^{2} g_{1,k}(\hat{x}, y, \hat{z})\, \rho_{xz}^{k}) \sqrt{ \sum_{k=2}^{4} g_{2,k}(\hat{x}, y, \hat{z})\, \rho_{xz}^{k-2}}}.
            \end{aligned}
        \end{equation*}
        By dividing the numerator and the denominator by $y^3$, we obtain
        \begin{equation*}
            \begin{aligned}
                T_1(x,y,z) &= \frac{\frac{(\sum_{k=2}^{3} f_k(\hat{x}, y, \hat{z})\, \rho_{xz}^{k-1})}{y^2} \sqrt{(1 + 2\rho_{xz} \hat{x})(\frac{1}{y^2}+2\frac{1}{y})} }{\frac{( \sum_{k=0}^{2} g_{1,k}(\hat{x}, y, \hat{z})\, \rho_{xz}^{k})}{y^2} \sqrt{ \frac{\sum_{k=2}^{4} g_{2,k}(\hat{x}, y, \hat{z})\, \rho_{xz}^{k-2}}{y^2}}}
                \\
                &= \frac{((a_1\hat{x}^2 + b_2 \hat{x} \hat{z}) \rho_{xz} + \frac{f_3 (\hat{x}, y, \hat{z} )}{y^2}\rho_{xz}^2)\sqrt{(1 + 2\rho_{xz} \hat{x})(\frac{1}{y^2}+2\frac{1}{y})}}{(1 + \frac{\sum_{k=1}^{2} g_{1,k}(\hat{x}, y, \hat{z})\, \rho_{xz}^{k}}{y^2} ) \sqrt{a_1 \hat{x}^2 + a_3 \hat{z}^2 + 2b_2 \hat{x} \hat{z} + \frac{\sum_{k=3}^{4} g_{2,k}(\hat{x}, y, \hat{z})\, \rho_{xz}^{k-2}}{y^2}}}.
            \end{aligned}
        \end{equation*}
        We observe that the numerator's lowest-order term in $\rho_{xz}$ is $(a_1\hat{x}^2 + b_2 \hat{x} \hat{z}) \rho_{xz}$. Conversely, the denominator has a zeroth-order term  $\sqrt{a_1 \hat{x}^2 + a_3 \hat{z}^2 + 2b_2 \hat{x} \hat{z}}$.
        Because $f_k, g_{1,k}, g_{2,k}$ are continuous with respect to $\hat{x}, \hat{z}$, and under the condition $y \geq \delta_0$, we deduce that for any $\epsilon > 0$, there exists a positive constant $\delta_2$ such that for all $(x, y, z) \in B((0, \bar{y}, 0), \delta_2)$, the following inequality holds
        \begin{equation*}
            T_1(x,y,z) \leq \epsilon.
        \end{equation*}
        For $T_2(x,y,z)$, we have
        \begin{equation}
            \label{equation:page_18_1}
            \begin{aligned}
                T_2 &= \frac{(\sum_{k=2}^{3} f_k(\hat{x}, y, \hat{z})\, \rho_{xz}^{k-1}) q(x,y,z) }{( \sum_{k=0}^{2} g_{1,k}(\hat{x}, y, \hat{z})\, \rho_{xz}^{k}) \sqrt{ \sum_{k=2}^{4} g_{2,k}(\hat{x}, y, \hat{z})\, \rho_{xz}^{k-2}}}
                \\
                & = \frac{(\sum_{k=2}^{3} f_k(\hat{x}, y, \hat{z})\, \rho_{xz}^{k-1}) ((1+\varphi_{ij})\rho_{xz}\hat{x}y + \rho_{xz}\hat{x} + y + 1 - \sqrt{4\rho_{xz}\hat{x}y+2\rho_{xz}\hat{x}+2y+1}\,) }{( \sum_{k=0}^{2} g_{1,k}(\hat{x}, y, \hat{z})\, \rho_{xz}^{k}) \sqrt{ \sum_{k=2}^{4} g_{2,k}(\hat{x}, y, \hat{z})\, \rho_{xz}^{k-2}}} \\
                & = \frac{\frac{(\sum_{k=2}^{3} f_k(\hat{x}, y, \hat{z})\, \rho_{xz}^{k-1})}{y^2} ((1+\varphi_{ij})\rho_{xz}\hat{x} + \rho_{xz}\frac{\hat{x}}{y} + 1 + \frac{1}{y} - \sqrt{4\rho_{xz}\hat{x}\frac{1}{y}+2\rho_{xz}\frac{\hat{x}}{y^2}+2\frac{1}{y}+\frac{1}{y^2}}\,) }{(1 + \frac{\sum_{k=1}^{2} g_{1,k}(\hat{x}, y, \hat{z})\, \rho_{xz}^{k}}{y^2} ) \sqrt{a_1 \hat{x}^2 + a_3 \hat{z}^2 + 2b_2 \hat{x} \hat{z} + \frac{\sum_{k=3}^{4} g_{2,k}(\hat{x}, y, \hat{z})\, \rho_{xz}^{k-2}}{y^2}}}.
            \end{aligned}
        \end{equation}
        For the numerator of the above equality, we get that for any $\epsilon > 0$ there exists $\delta_2 > 0$, if $\rho_{xz} < \delta_2$, then
        \begin{equation*}
            \begin{aligned}
                &(1+\varphi_{ij})\rho_{xz}\hat{x} + \rho_{xz}\frac{\hat{x}}{y} + 1 + \frac{1}{y} - \sqrt{4\rho_{xz}\hat{x}\frac{1}{y}+2\rho_{xz}\frac{\hat{x}}{y^2}+2\frac{1}{y}+\frac{1}{y^2}}
                \\
                &\leq 1+\frac{1}{y} - \sqrt{2\frac{1}{y} + \frac{1}{y^2}} + \epsilon
                \\
                & \leq 1 + \epsilon.
            \end{aligned}
        \end{equation*}
        So the numerator's lowest-order term of (\ref{equation:page_18_1}) in $\rho_{xz}$ is also one order. By analogous analysis we can get similar results, i.e.,
    \begin{equation}
        \label{equation:proof_delta_2_estimate}
        \forall \epsilon > 0, \exists \delta_2 > 0, \forall \bar{z} > \delta_0, \forall (x,y,z) \in B((0,0,\bar{z}), \delta_2), T_2(x,y,z) \leq \epsilon.
    \end{equation}
        \begin{itemize}
            \item $y \to 0, z \to 0, x \to \bar{x} \geq \delta_0$
        \end{itemize}
        Let $\rho_{yz} = \sqrt{y^2 + z^2}$ and $y = \rho_{yz} \hat{y} , z= \rho_{yz} \hat{z}$. Since the region is in $\mathbb{R}_{>0}^3\cap\mathbb{S}^2$, we have $0 < \hat{y} <1$ and $0 < \hat{z} <1$. We express both the numerator and the denominator as polynomials in the variable $\rho_{yz}$
        \begin{equation}
        \begin{aligned}
            f(x, y, z) &= \sum_{k=2}^{3} f_k(x, \hat{y}, \hat{z})\, \rho_{yz}^k, \quad
            g_1(x, y, z) = \sum_{k=0}^{2} g_{1,k}(x, \hat{y}, \hat{z})\, \rho_{yz}^k, \quad
            g_2(x, y, z) = \sum_{k=2}^{4} g_{2,k}(x, \hat{y}, \hat{z})\, \rho_{yz}^k,
        \end{aligned}
    \end{equation}
    where
    \begin{equation}
        \begin{aligned}
            f_2(x, \hat{y}, \hat{z}) &= a_1 x^2 \hat{y}^2 + b_3 x^2 \hat{y} \hat{z}, \\
            f_3(x, \hat{y}, \hat{z}) &= (a_1 + b_2 + b_3)x^2 \hat{y}^2 \hat{z} + b_2 x \hat{y}^2 \hat{z}, \\
            \end{aligned}
        \end{equation}
        and
        \begin{equation}
            \begin{aligned}
                g_{1,0}(x, \hat{y}, \hat{z}) &= x^2, \\
                g_{1,1}(x, \hat{y}, \hat{z}) &= 2(1+\varphi_{ij})x^2\hat{y} + 2\varphi_{ij}x\hat{y}, \\
                g_{1,2}(x, \hat{y}, \hat{z}) &= (3+4\varphi_{ij}+\varphi_{ij}^2)x^2\hat{y}^2 + 2(1+\varphi_{ij})x\hat{y}^2 + \hat{y}^2, \\
            \end{aligned}
        \end{equation}
        and
        \begin{equation}
            \begin{aligned}
                g_{2,2}(x, \hat{y}, \hat{z}) &= a_1x^2\hat{y}^2 + a_2x^2\hat{z}^2 + 2b_3x^2\hat{y}\hat{z}, \\
                g_{2,3}(x, \hat{y}, \hat{z}) &= 2(a_1+b_2+b_3)x^2\hat{y}^2\hat{z} + 2(a_2+b_1+b_3)x^2\hat{y}\hat{z}^2 + 2b_2x\hat{y}^2\hat{z} + 2b_1x\hat{y}\hat{z}^2, \\
                g_{2,4}(x, \hat{y}, \hat{z}) &= a_1 x^2 \hat{y}^2 + 2(a_1 + b_2 + b_3)x^2 \hat{y}^2 \hat{z} + 2(c + b_1 + b_2 + b_3)x^2 \hat{y}^2 \hat{z}^2.
            \end{aligned}
        \end{equation}
         By the condition $\langle \star\rangle$, we have $a_1 + a_2 + b_3 = 1 - \varphi_{ij}^2 + 1 - \varphi_{jk}^2 + \varphi_{ik} + \varphi_{ij} \varphi_{jk} > 0$.  Hence we have $g_{2,2} > 0$. Dividing both the numerator and the denominator of $T_1(x,y,z)$ by $\rho_{yz}$ gives
        \begin{equation*}
            \begin{aligned}
                T_1(x,y,z) = \frac{(\sum_{k=2}^{3} f_k(x, \hat{y}, \hat{z})\, \rho_{yz}^{k-1}) \sqrt{(1 + 2x)(1+2\rho_{yz} \hat{y})}}{( \sum_{k=0}^{2} g_{1,k}(x, \hat{y}, \hat{z})\, \rho_{yz}^{k}) \sqrt{ \sum_{k=2}^{4} g_{2,k}(x, \hat{y}, \hat{z})\, \rho_{yz}^{k-2}}}.
            \end{aligned}
        \end{equation*}
        By dividing the numerator and the denominator of $T_1(x,y,z)$ by $x^3$, we obtain
        \begin{equation*}
            \begin{aligned}
                T_1(x,y,z) &= \frac{((a_1 \hat{y}^2 + b_3 \hat{y} \hat{z}) \rho_{yz} +  \frac{f_3(x, \hat{y}, \hat{z})}{x^2}\, \rho_{yz}^{2}) \sqrt{(\frac{1}{x^2} + 2\frac{1}{x})(1+2\rho_{yz} \hat{y})}}{( 1 + \sum_{k=1}^{2} \frac{g_{1,k}(x, \hat{y}, \hat{z})\, }{x^2}\rho_{yz}^{k}) \sqrt{ a_1 \hat{y}^2+ a_2 \hat{z}^2 + 2b_3 \hat{y} \hat{z} + \sum_{k=3}^{4} \frac{g_{2,k}(x, \hat{y}, \hat{z})}{x^2}\, \rho_{yz}^{k-2}}}.
            \end{aligned}
        \end{equation*}
        Dividing the numerator and the denominator of $T_2(x,y,z)$ by $x^3$, we get
        \begin{equation*}
            T_2 (x,y,z) = \frac{(\sum_{k=2}^{3} f_k(x, \hat{y}, \hat{z})\, \rho_{yz}^{k-1}) ((1+\varphi_{ij})\rho_{yz}\hat{y} + \rho_{yz}\frac{\hat{y}}{x} + 1 + \frac{1}{x} - \sqrt{4\rho_{yz}\hat{y}\frac{1}{x}+2\rho_{yz}\frac{\hat{y}}{x^2}+2\frac{1}{x}+\frac{1}{x^2}}\,.) }{( \sum_{k=0}^{2} g_{1,k}(x, \hat{y}, \hat{z})\, \rho_{yz}^{k}) \sqrt{ \sum_{k=2}^{4} g_{2,k}(x, \hat{y}, \hat{z})\, \rho_{yz}^{k-2}}}.
        \end{equation*}
        Analogous to the case where $ x \to 0, z \to 0, y \to \bar{y} \geq \delta_0 $, we obtain
        \begin{equation*}
            \forall \epsilon >0, \exists \delta_3 > 0, \forall \bar{x} \geq \delta_0, \forall (x,y,z) \in B((\bar{x},0,0), \delta_3), T_1(x,y,z) \leq \epsilon, T_2(x,y,z) \leq \epsilon.
        \end{equation*}

        Let $\bar{\delta} = \min \{\delta_0, \delta_1, \delta_2, \delta_3\}$.
        It remains to analyze the behavior in the regions defined by $\{x = 0, y \geq \delta, z \geq \delta\} \cup \{y = 0, x \geq \delta, z \geq \delta\} \cup \{z = 0, x \geq \delta, y \geq \delta\}$.
        These are the cases where, on the boundary, precisely one variable vanishes. The discussion for these three cases is presented below.

     \textbf{\text{Case 3:} One variable approaches zero}

        \begin{itemize}
            \item $ x \to 0, y \to \bar{y} \geq \bar{\delta}, z \to \bar{z} \geq \bar{\delta} $
        \end{itemize}

        We now express both the numerator and the denominator as polynomials in the variable $x$
        \begin{equation}
            \begin{aligned}
                f(x, y, z) = \sum_{k=1}^{2} f_k(y, z)\, x^k, \quad
                g_1(x, y, z) = \sum_{k=0}^{2} g_{1,k}(y, z)\, x^k, \quad
                g_2(x, y, z) = \sum_{k=0}^{2} g_{2,k}(y, z)\, x^k,
            \end{aligned}
        \end{equation}
        where
        \begin{equation}
            \begin{aligned}
                &f_1(y, z) = b_2 y^2 z, \\
                &f_2(y, z) = (a_1 + b_2 + b_3)y^2 z + b_3 y z + a_1 y^2,
         	\end{aligned}
        \end{equation}
        and
        \begin{equation}
            \begin{aligned}
                &g_{1,0}(y, z) = y^2, \\
                &g_{1,1}(y, z) = 2(1 + \varphi_{ij})y^2 + 2\varphi_{ij}y, \\
                &g_{1,2}(y, z) = (3 + 4\varphi_{ij} + \varphi_{ij}^2)y^2 + 2(1 + \varphi_{ij})y + 1,
         	\end{aligned}
        \end{equation}
        and
        \begin{equation}
            \begin{aligned}
                g_{2,0}(y, z) &= a_3 y^2 z^2,\\
                g_{2,1}(y, z) &= 2(a_3 + b_1 + b_2)y^2 z^2 + 2b_2 y^2 z + 2b_1 y z^2, \\
                g_{2,2}(y, z) &= 2(a_1 + b_2 + b_3)y^2 z + 2(a_2 + b_1 + b_3)y z^2 + a_1 y^2 + a_2 z^2
                \\
                & \quad + 2b_3 y z + 2(c + b_1 + b_2 + b_3)y^2 z^2.
            \end{aligned}
        \end{equation}
        By dividing both the numerator and the denominator of $T_1(x,y,z)$ by $y^3 z$, we obtain
        \begin{equation*}
            T_{1}(x,y,z) = \frac{(b_2 x + \frac{f_2(y,z)}{y^2z}x^2) \sqrt{(1+2x)(\frac{1}{y^2} + 2\frac{1}{y})}}{(1+ \sum_{k=1}^2\frac{g_{1,k}(y,z)}{y^2}x^{k}) \sqrt{a_3 + \sum_{k=1}^2\frac{g_{2,k}(y,z)}{y^2 z^2}x^k}}.
        \end{equation*}
        Analyzing these coefficients gives the following cases

        \begin{itemize}
            \item[(1)] If $ a_3 \ne 0 $, then for a small enough $\epsilon > 0$ there exists $\delta_{4} > 0$, if $x < \delta_4$, the denominator satisfies
            \begin{equation*}
                (\frac{1}{z^2}+ \sum_{k=1}^2\frac{g_{1,k}(y,z)}{y^2z^2}x^{k}) \sqrt{a_3 + \sum_{k=1}^2\frac{g_{2,k}(y,z)}{y^2 z^2}x^k} \geq {\sqrt{a_3}} - \epsilon,
            \end{equation*} while the numerator satisfies
            \begin{equation*}
                {(b_2 x + \frac{f_2(y,z)}{y^2z}x^2) \sqrt{(1+2x)(\frac{1}{y^2} + 2\frac{1}{y})}} \leq \epsilon.
            \end{equation*}
            Hence, when $x \leq \delta_4$, $ T_1(x, y, z) \leq \frac{\epsilon}{\sqrt{a_3} - \epsilon}$ for some small enough $\epsilon > 0$ in this case.

            \item[(2)] If $ a_3 = 0 $, we have $g_{2,0} \equiv 0$ and by (\ref{Equation2_a_b_c_definition}) we obtain that $ b_1 = b_2 $. When $ b_1 \ne 0 $, by dividing both the numerator and the denominator of $T_1(x,y,z)$ by $x^{\frac{1}{2}}$, we obtain
            \begin{equation*}
                \begin{aligned}
                    T_{1}(x,y,z) = \frac{(b_2 x^{\frac{1}{2}} + \frac{f_2(y,z)}{y^2z}x^{\frac{3}{2}}) \sqrt{(1+2x)(\frac{1}{y^2} + 2\frac{1}{y})}}{(1+ \sum_{k=1}^2\frac{g_{1,k}(y,z)}{y^2}x^{k}) \sqrt{2(b_1 + b_2)+ 2b_2 \frac{1}{z} + 2b_1 \frac{1}{y} + \frac{g_{2,2}(y,z)}{y^2 z^2}x}}.
                \end{aligned}
            \end{equation*}
            Then for a small enough $\epsilon > 0$ there exists $\delta_{4} > 0$, if $x < \delta_4$, the denominator satisfies
            \begin{equation*}
                {(1+ \sum_{k=1}^2\frac{g_{1,k}(y,z)}{y^2}x^{k}) \sqrt{2(b_1 + b_2)+ 2b_2 \frac{1}{z} + 2b_1 \frac{1}{y} + \frac{g_{2,2}(y,z)}{y^2 z^2}x}} \geq 2 b_1,
            \end{equation*}
            while the numerator satisfies
            \begin{equation*}
                {(b_2 x^{\frac{1}{2}} + \frac{f_2(y,z)}{y^2z}x^{\frac{3}{2}}) \sqrt{(1+2x)(\frac{1}{y^2} + 2\frac{1}{y})}} \leq \epsilon.
            \end{equation*}
            Hence, when $x \leq \delta_4$, $ T_1(x, y, z) \leq \frac{\epsilon}{2b_1}$ for some small enough $\epsilon > 0$ in this case.

            \item[(3)] If $ a_3 = 0 $ and $ b_1 = b_2 = 0 $, we have $f_1 \equiv 0, g_{2,0} \equiv 0$ and $g_{2,1} \equiv 0$. We divide both the numerator and the denominator of $T_1(x,y,z)$ by $x$, and obtain
            \begin{equation}
            	T_{1}(x,y,z) = \frac{((a_1 + b_2 + b_3)+  b_3 \frac{1}{y} + a_1 \frac{1}{z})x \sqrt{(1+2x)(\frac{1}{y^2} + 2\frac{1}{y})}}{(1+ \sum_{k=1}^2\frac{g_{1,k}(y,z)}{y^2}x^{k}) \sqrt{\frac{g_{2,2}(y,z)}{y^2 z^2}}}.
            \end{equation}
            By analogous analysis, we get that when $x \leq \delta_4$, $ T_1(x, y, z) \leq \epsilon$ for some small enough $\epsilon > 0$ in this case.
        \end{itemize}

        For $T_2(x,y,z)$, dividing both numerator and denominator by $y^3z$ gives
        \begin{equation*}
            T_2(x,y,z) = \frac{(b_2 x + \frac{f_2(y,z)}{y^2z}x^2) ((1+\varphi_{ij}) x + \frac{x}{y} + 1 + \frac{1}{y} - \sqrt{4\frac{x}{y} + 2\frac{x}{y^2} + 2\frac{1}{y} + \frac{1}{y^2}})}{(1+ \sum_{k=1}^2\frac{g_{1,k}(y,z)}{y^2}x^{k}) \sqrt{a_3 + \sum_{k=1}^2\frac{g_{2,k}(y,z)}{y^2 z^2}x^k}}.
        \end{equation*}
        We observe that the numerator of $T_2(x,y,z)$ remains arbitrarily small under the three previously discussed cases. In summary, we have
        \begin{equation*}
            \forall \epsilon > 0, \exists \delta_4 > 0, \forall \bar{y} \geq \bar{\delta}, \bar{z} \geq \bar{\delta}, \forall (x,y,z) \in B((0,\bar{y},\bar{z}), \delta_4), T_1(x,y,z) \leq \epsilon, T_2(x,y,z) \leq \epsilon.
        \end{equation*}

        \begin{itemize}
            \item $ y \to 0, x \to \bar{x} \geq \bar{\delta}, z \to \bar{z} \geq \bar{\delta} $
        \end{itemize}

        We express both the numerator and the denominator as polynomials in the variable $y$
        \begin{equation}
        \begin{aligned}
            f(x, y, z) &= \sum_{k=1}^{2} f_k(x, z)\, y^k, \quad
            g_1(x, y, z) = \sum_{k=0}^{2} g_{1,k}(x, z)\, y^k, \quad
            g_2(x, y, z) = \sum_{k=0}^{2} g_{2,k}(x, z)\, y^k,
        \end{aligned}
    \end{equation}
    where
    \begin{equation}
        \begin{aligned}
            &f_1(x, z) = b_3 x^2 z, \\
            &f_2(x, z) = (a_1 + b_2 + b_3)x^2 z + b_2 x z + a_1 x^2,
    	\end{aligned}
    \end{equation}
    and
    \begin{equation}
        \begin{aligned}
            &g_{1,0}(x, z) = x^2, \\
            &g_{1,1}(x, z) = 2(1 + \varphi_{ij})x^2 + 2\varphi_{ij}x, \\
            &g_{1,2}(x, z) = (3 + 4\varphi_{ij} + \varphi_{ij}^2)x^2 + 2(1 + \varphi_{ij})x + 1,
    	\end{aligned}
    \end{equation}
    and
    \begin{equation}
    	\begin{aligned}
    		g_{2,0}(x, z) &= a_2 x^2 z^2, \\
    		g_{2,1}(x, z) &= 2(a_2 + b_1 + b_3)x^2 z^2 + 2b_3 x^2 z + 2b_1 x z^2, \\
    		g_{2,2}(x, z) &= 2(a_1 + b_2 + b_3)x^2 z + 2(a_3 + b_1 + b_2)x z^2 + a_1 x^2 + a_3 z^2 \\
    		&\quad + 2b_2 x z + 2(c + b_1 + b_2 + b_3)x^2 z^2.
    	\end{aligned}
    \end{equation}
    Observing the coefficients, we note a symmetry with the case $ x \to 0, y \to \bar{y} \geq \bar{\delta}, z \to \bar{z} \geq \bar{\delta} $. Hence, applying the analogous method we have
    \begin{equation*}
        \forall \epsilon > 0, \exists \delta_5 > 0, \forall \bar{x} \geq \bar{\delta}, \forall \bar{z} \geq \bar{\delta}, \forall (x,y,z) \in B((\bar{x},0,\bar{z}), \delta_5), T_1(x,y,z) \leq \epsilon, T_2(x,y,z) \leq \epsilon.
    \end{equation*}

    \begin{itemize}
        \item  $ z \to 0, x \to \bar{x} \geq \bar{\delta}, z \to \bar{y} \geq \bar{\delta} $
    \end{itemize}
    We express both the numerator and the denominator as polynomials in the variable $z$
        \begin{equation}
        \begin{aligned}
            f(x, y, z) = \sum_{k=0}^{1} f_k(x, y)\, z^k, \quad
            g_1(x, y, z) = g_{1,0}(x, y), \quad
            g_2(x, y, z) = \sum_{k=0}^{2} g_{2,k}(x, y)\, z^k,
        \end{aligned}
    \end{equation}
    where
        \begin{equation}
            \begin{aligned}
                f_0(x, y) &= a_1 x^2 y^2, \\
                f_1(x, y) &= (a_1 + b_2 + b_3)x^2 y^2 + b_2 x y^2 + b_3 x^2 y, \quad
        	 \end{aligned}
     	\end{equation}	
     	and
     	\begin{equation}
     	        \begin{aligned}
                g_{1,0}(x, y) = (3 + 4\varphi_{ij} + \varphi_{ij}^2)x^2 y^2 + 2(1 + \varphi_{ij})(x^2 y + x y^2) + x^2 + y^2 + 2\varphi_{ij}xy,
       \end{aligned}
           \end{equation}
           and
       \begin{equation}
               \begin{aligned}
                g_{2,0}(x, y) &= a_1 x^2 y^2, \\
                g_{2,1}(x, y) &= 2(a_1 + b_2 + b_3)x^2 y^2 + 2b_2 x y^2 + 2b_3 x^2 y, \\
                g_{2,2}(x, y) &= 2(a_3 + b_1 + b_2)x y^2 + 2(a_2 + b_1 + b_3)x^2 y + a_3 y^2 + a_2 x^2\\
                 &\quad + 2b_1 x y + 2(c + b_1 + b_2 + b_3)x^2 y^2.
            \end{aligned}
        \end{equation}
        By dividing both the numerator and the denominator of $T_1(x,y,z)$ by $x^3 y^3$, we obtain
        \begin{equation*}
            T_1(x,y,z) = \frac{(a_1 + \frac{f_1(x,y)}{x^2 y^2 } z)\sqrt{(\frac{1}{x^2} + 2\frac{1}{x})(\frac{1}{y^2} + 2\frac{1}{y})}}{(3 + 4\varphi_{ij} + \varphi_{ij}^2 + 2(1+\varphi_{ij})(\frac{1}{x} + \frac{1}{y})+ \frac{1}{y^2} + \frac{1}{x^2} + 2\varphi_{ij} \frac{1}{x} \frac{1}{y}) \sqrt{a_1 + \frac{g_{2,1}(x,y)}{x^2 y^2}z + \frac{g_{2,2}(x,y)}{x^2 + y^2}z^2}}.
        \end{equation*}
        Dividing both the numerator and the denominator of $T_2(x,y,z)$ by $x^3 y^3$, we obtain
        \begin{equation*}
            T_2(x,y,z) = \frac{(b_2 x + \frac{f_2(x,y)}{x^2 y^2}x^2) (1+\varphi_{ij}+ \frac{1}{y} + \frac{1}{x} + \frac{1}{xy} - \sqrt{4\frac{1}{xy} + 2\frac{1}{xy^2} + 2\frac{1}{yx^2} + \frac{1}{x^2y^2}})}{(1+ \sum_{k=1}^2\frac{g_{1,k}(x,y)}{x^2}x^{k}) \sqrt{a_1 + \sum_{k=1}^2\frac{g_{2,k}(x,y)}{x^2 y^2}x^k}}.
        \end{equation*}
        Analyzing these coefficients, we obtain the following cases.

       \noindent (1) If $ a_1 \ne 0 $, then for any $\epsilon > 0$, there exits a constant $\delta_6 > 0$ such that if $(x,y,z) \in B((\bar{x},\bar{y}, 0), \delta_6)$ for any $\bar{y} \geq \bar{\delta}, \bar{z} \geq \bar{\delta}$, it follows that
       \begin{equation*}
       \begin{aligned}
              		T_1(x,y,z) &\leq  \frac{a_1\sqrt{(\frac{1}{\bar{x}^2} + 2\frac{1}{\bar{x}})(\frac{1}{\bar{y}^2} + 2\frac{1}{\bar{y}})}}{(3 + 4\varphi_{ij} + \varphi_{ij}^2 + 2(1+\varphi_{ij})(\frac{1}{\bar{x}} + \frac{1}{\bar{y}})+ \frac{1}{\bar{y}^2} + \frac{1}{\bar{x}^2} + 2\varphi_{ij} \frac{1}{\bar{x}} \frac{1}{\bar{y}}) \sqrt{a_1}} + \epsilon
              		\\
              		&\leq  \frac{a_1\sqrt{(\frac{1}{\bar{x}^2} + 2\frac{1}{\bar{x}})(\frac{1}{\bar{y}^2} + 2\frac{1}{\bar{y}})}}{(3 + 4\varphi_{ij} + \varphi_{ij}^2)\sqrt{a_1} } + \epsilon
              		\\
              		&\leq  \frac{a_1\sqrt{(\frac{1}{\bar{\delta}^2} + 2\frac{1}{\bar{\delta}})(\frac{1}{\bar{\delta}^2} + 2\frac{1}{\bar{\delta}})}}{(3 + 4\varphi_{ij} + \varphi_{ij}^2)\sqrt{a_1} } + \epsilon,
       \end{aligned}
       \end{equation*}
       and
       \begin{equation*}
            T_2(x,y,z) \leq \frac{a_1 (1+ \varphi_{ij} + \frac{2}{\bar{\delta}} + \frac{1}{\bar{\delta}^2} + \sqrt{4\frac{1}{\bar{\delta}^2} + 4\frac{1}{\bar{\delta}^3} + \frac{1}{\bar{\delta}^4}})}{(3 + 4\varphi_{ij} + \varphi_{ij}^2)\sqrt{a_1} } + \epsilon.
       \end{equation*}
       Hence, let $$C =  \max\{\frac{a_1\sqrt{(\frac{1}{\bar{\delta}^2} + 2\frac{1}{\bar{\delta}})(\frac{1}{\bar{\delta}^2} + 2\frac{1}{\bar{\delta}})}}{(3 + 4\varphi_{ij} + \varphi_{ij}^2)\sqrt{a_1} } + 1, \frac{a_1 (1+ \varphi_{ij} + \frac{2}{\bar{\delta}} + \frac{1}{\bar{\delta}^2} + \sqrt{4\frac{1}{\bar{\delta}^2} + 4\frac{1}{\bar{\delta}^3} + \frac{1}{\bar{\delta}^4}})}{(3 + 4\varphi_{ij} + \varphi_{ij}^2)\sqrt{a_1} } + 1\},$$
        the result in this case follows.

    	(2) If $ a_1 = 0 $, we have $f_0 \equiv 0$ and $g_{2,0} \equiv 0$ and by (\ref{Equation2_a_b_c_definition}) we get that $ b_2 = b_3 $. When $ b_2 \ne 0 $, by dividing both the numerator and the denominator of $T_1(x,y,z)$ by $z^{1/2}$, we obtain
    	\begin{equation*}
    	T_1(x,y,z) = \frac{(a_1 + b_2 + b_3 + b_2 \frac{1}{x} + b_3 \frac{1}{y}) z^{\frac{1}{2}})\sqrt{(\frac{1}{x^2} + 2\frac{1}{x})(\frac{1}{y^2} + 2\frac{1}{y})}}{(3 + 4\varphi_{ij} + \varphi_{ij}^2 + 2(1+\varphi_{ij})(\frac{1}{x} + \frac{1}{y})+ \frac{1}{y^2} + \frac{1}{x^2} + 2\varphi_{ij} \frac{1}{x} \frac{1}{y}) \sqrt{ \frac{g_{2,1}(x,y)}{x^2 y^2} + \frac{g_{2,2}(x,y)}{x^2 + y^2}z^2}},
    	\end{equation*}
    	 then for a small enough $\epsilon > 0$ there exists $\delta_{6} > 0$, if $(x,y,z) \in B((\bar{x},\bar{y}, 0), \delta_6)$, the denominator satisfies
         \begin{equation*}
         \begin{aligned}
                      &{(3 + 4\varphi_{ij} + \varphi_{ij}^2 + 2(1+\varphi_{ij})(\frac{1}{x} + \frac{1}{y})+ \frac{1}{y^2} + \frac{1}{x^2} + 2\varphi_{ij} \frac{1}{x} \frac{1}{y}) \sqrt{ \frac{g_{2,1}(x,y)}{x^2 y^2} + \frac{g_{2,2}(x,y)}{x^2 + y^2}z^2}} \\
                      &
                      \geq (3 + 4\varphi_{ij} + \varphi_{ij}^2)\sqrt{4b_2^2},
         \end{aligned}
         \end{equation*}
         while the numerator satisfies
         \begin{equation*}
             {(a_1 + b_2 + b_3 + b_2 \frac{1}{x} + b_3 \frac{1}{y}) z^{\frac{1}{2}}\sqrt{(\frac{1}{x^2} + 2\frac{1}{x})(\frac{1}{y^2} + 2\frac{1}{y})}} \leq \epsilon.
         \end{equation*}
         An analogous argument applies to the numerator of $T_2(x,y,z)$.
         Hence, there exists $\delta_{7} > 0$ such that if $(x,y,z) \in B((\bar{x},\bar{y}, 0), \delta_7)$, we have $ T_1(x, y, z) \leq \epsilon, T_2(x,y,z) \leq \epsilon$ for some small enough $\epsilon > 0$.
    	
    	(3) If $ a_1 = 0 $ and $ b_2 = b_3 = 0 $, we have that $f(x,y,z) \equiv 0$ which implies that $T_1(x,y,z) \equiv 0$ and $T_2(x,y,z) \equiv 0$. Hence the result in this case follows.
    	
    	Let $\delta = \min\{\delta_0, \delta_1, \delta_2, \delta_3, \delta_4, \delta_5, \delta_6\}$.
    	 Summarizing the preceding cases, we conclude that there exists $C_1(\varphi_{ij}, \varphi_{ik}, \varphi_{jk})>0$ and $C_2(\varphi_{ij}, \varphi_{ik}, \varphi_{jk})>0$ such that if $\min\{x,y,z\} \leq \delta$, then $T_1(x,y,z) < C_1(\varphi_{ij}, \varphi_{ik}, \varphi_{jk})$ and $T_2(x,y,z) < C_2(\varphi_{ij}, \varphi_{ik}, \varphi_{jk})$.

    Based on all the discussions of $\Omega_1, \Omega_2$, we conclude that $T_1(x, y, z)$ and $T_2(x, y, z)$ are uniformly bounded above by positive constants depending only on $\Phi$, which are denoted by
    \begin{align*}
        M_1=M_1(\Phi) &= \max_{\triangle_{ijk} \in F} \big\{C_1(\varphi_{ij}, \varphi_{ik}, \varphi_{jk}),  \frac{5}{128 a^{14}(1+\bar{\varphi})^2\sqrt{1+\bar{\varphi}} }\big\}, \\
        M_2=M_2(\Phi) &=  \max_{\triangle_{ijk} \in F} \big\{C_2(\varphi_{ij}, \varphi_{ik}, \varphi_{jk}),  \frac{5}{64 a^{14}(1+\bar{\varphi})^2\sqrt{1+\bar{\varphi}} }\big\}.
    \end{align*}
    By (\ref{AB-equality}) and Corollary \ref{Corollary:Glickenstein-corollary} we have
    $$
    A_i = \sum_{v_j \sim v_i} B_{ij}(\cosh l_{ij} - 1), \quad
    B_{ij} = \frac{\partial \theta_i^{jk}}{\partial u_j} + \frac{\partial \theta_i^{jl}}{\partial u_j},
    $$
    where $\theta_i^{jk}$ and $\theta_i^{jl}$ are inner angles in triangles $\triangle_{ijk}$ and $\triangle_{ijl}$ respectively. Let $d = \max_{v_i \in V} \{\deg(v_i)\}$, we obtain
    \begin{align*}
        |B_{ij}| &\le \left| \frac{\partial \theta_{i}^{jk}}{\partial u_{j}} \right| + \left| \frac{\partial \theta_{i}^{jl}}{\partial u_{j}} \right| \le M_1 + M_1 = 2 M_1, \\
        |A_i| &= \left| \sum_{v_j \sim v_i} (\frac{\partial \theta_{j}^{ik}}{\partial u_{i}} + \frac{\partial \theta_{j}^{il}}{\partial u_{i}})(\cosh l_{ij}-1) \right| \le \sum_{v_j \sim v_i} (M_2 + M_2) = 2 d_i M_2 \le 2 d M_2.
    \end{align*}
    Thus, there exist positive constants $ C_1(\Phi) = 2d M_2 $ and $ C_2(\Phi) = 2 M_1 $, depending only on $\Phi$ and the combinatorial structure of the triangulation $\mathcal{T}=(V,E,F)$, such that the inequalities $0< A_i \leq C_1(\Phi) $ and $0< B_{ij} \leq C_2(\Phi) $ hold for all vertices $ v_i, v_j $ and any circle packing metric $ r \in \mathbb{R}_{>0}^N $. This complete the proof of Theorem \ref{Theorem2_uniform_upper_bound}. \hfill $\square$
\end{proof}

\section{Long time existence of the combinatorial $p$-th Calabi flows}

In this section, we give new proofs of the long time existences of solutions to the combinatorial Calabi flows established by Ge-Xu \cite{Ge2016}, Ge-Hua \cite{Ge2018} and the combinatorial $p$-th Calabi flows established by Lin-Zhang \cite{Lin2019} in $\mathbb{H}^2$, and prove Theorem \ref{long time existence}, based on our main results on two improved estimates for the discrete Laplace operator, namely Theorem \ref{main result1} and Theorem \ref{main result2}.

\begin{theorem}\label{longtime existence}
    For any initial circle packing metric $r(0) \in \mathbb{R}_{>0}^{N}$, the solutions to both the combinatorial Calabi flows (\ref{Equation:rewritten_calabi_equation}) and the combinatorial $p$-th Calabi flow (\ref{Equation:combinatorial_p_th_calabi_flow}) in $\mathbb{H}^{2}$ exist for all time $t \in [0,+\infty)$.
\end{theorem}

Before proving the theorem, we first present an important lemma established by Zhang and Zheng \cite{Zhang24} (see Lemma 4.1 in \cite{Zhang24}).
\begin{lemma}\label{lemma}
    Let $ \Delta v_{i} v_{j} v_{k} $ be a hyperbolic triangle which is patterned by three circles with the fixed weight $ \Phi_{i j}, \Phi_{i k}, \Phi_{j k} \in[0, \pi) $ satisfying $ (\star) $. Let $ \theta_{i}^{j k} $ be the inner angle at $ v_{i} $. Then $ \forall \epsilon>0 $, there exists a number $ l $ so that when $ r_{i}>l $, the inner angle $ \theta_{i}^{j k} $ is smaller than $ \epsilon $.
\end{lemma}

\textbf{Proof of Theorem \ref{longtime existence}}  Note that the combinatorial Calabi flow (\ref{Equation:rewritten_calabi_equation}) is precisely the combinatorial $p$-th Calabi flow (\ref{Equation:combinatorial_p_th_calabi_flow}) in the case of $p=2$, hence we only need to consider the combinatorial $p$-th Calabi flow (\ref{Equation:combinatorial_p_th_calabi_flow}).  We begin by showing that the solution to the $p$-th Calabi flow (\ref{Equation:combinatorial_p_th_calabi_flow}) admits a positive lower bound on any finite time interval $[0,T)$ for $0 \leq T < +\infty$.

    Let $ d_{i} $ denote the degree at vertex $ v_{i} $, which is the number of edges adjacent to $ v_{i} $. Set $ d=\max \left\{d_{1}, \ldots, d_{N}\right\} $, by the definition of the combinatorial Gauss curvature
        \begin{equation*}
    K_i = 2\pi - \sum_{\triangle_{ijk} \in F} \theta_{i}^{jk},
\end{equation*}
 then we have $ (2-d) \pi<K_{i}<2 \pi $, hence
    \begin{equation*}
        \left|K_{j}-K_{i}\right|<d \pi, \forall v_j \sim v_i.
    \end{equation*}
    By (\ref{Equation:combinatorial_p_th_calabi_flow}) and Theorem \ref{main result2}, we obtain
    \begin{equation*}
        \begin{aligned}
            |\frac{d u_i}{dt}| & = |\sum_{v_j \sim v_i} B_{ij}|K_j-K_i|^{p-2}(K_j-K_i) +A_i K_i|\\
                               & \leq \sum_{v_j \sim v_i} B_{ij}|K_j-K_i|^{p-1} + A_i |K_i| \\
                               & \leq d C_2 (d\pi)^{p-1} + C_1(d+2)\pi  \triangleq C,
        \end{aligned}
    \end{equation*}
    where $C$ is a positive constant depending only on the triangulation $\mathcal{T}$. Thus, we have
    \begin{equation*}
        |u_i(t)| \leq |u_i(0)| + C t, \forall i=1, 2, \ldots, N.
    \end{equation*}
    This implies that \(u_i(t)\) is bounded above over any finite time interval, and since \(u_i=\ln \tanh \frac{r_i(t)}{2}\), it follows that \(r_i(t)\) is bounded below in finite time as follows
    \begin{equation*}
        0< \tanh (\frac{r_i(0)}{2}) e^{-Ct}\leq \tanh (\frac{r_i(t)}{2}) \leq \tanh (\frac{r_i(0)}{2}) e^{Ct}.
    \end{equation*}
    Thus we can get an estimate of the lower bound of $r_i(t)$
    \begin{equation*}
        r_{i}(t) \geq \frac{1}{2}\ln \frac{1+\tanh (\frac{r_i(0)}{2}) e^{-Ct}}{1-\tanh (\frac{r_i(0)}{2}) e^{-Ct}} > 0.
    \end{equation*}
    Next, we prove that the solution to the flow equation admits a uniform upper bound. For a vertex \(i\), by Lemma \ref{lemma} one may choose a sufficiently large positive constant \(l\) such that \(r_i>l\) and the interior angle \(\theta_i\) is less than \(\frac{\pi}{2d_i}\), where \(d_i\) denotes the degree of vertex \(i\). Hence, $\pi <K_i< 2\pi$, which leads to \(K_j-K_i<\pi\) for any \(v_j \sim v_i\). Note that if \(a\le b\) then \(|a|^{p-2}a\le|b|^{p-2}b\), which yields

    \begin{equation*}
        \begin{aligned}
        \frac{1}{\sinh r_i}\frac{dr_i}{dt} &= \sum_{v_j \sim v_i}B_{ij}\bigl|K_j-K_i\bigr|^{p-2}(K_j-K_i)-A_iK_i \\
        &\le \sum_{v_j \sim v_i}B_{ij}\pi^{p-1}-A_i\pi \\
        &=\pi^{p-1}\left(\sum_{v_j \sim v_i}B_{ij}-\frac{1}{\pi^{p-2}}A_i\right) \\
        &= \pi^{p-1}\left(\sum_{v_j \sim v_i}B_{ij}(1-\frac{\cosh l_{ij}-1}{\pi^{p-2}})\right),
        \end{aligned}
    \end{equation*}
    where the last equality is due to Corollary \ref{Corollary:Glickenstein-corollary}.
    Noting that $$ \cosh l_{ij} = \cosh r_i \cosh r_j + \cos \Phi_{ij} \sinh r_i \sinh r_j $$ and the condition \(0 \leq \Phi_{ij} < \pi\), it follows that
    \begin{equation*}
        \exists R > 0, \text{ such that } r_i > R, \text{ then } 1-\frac{\cosh l_{ij}-1}{\pi^{p-2}} < 0,
    \end{equation*}
    Since \(B_{ij} \geq 0\), when $r_i > \max\{R,l\}$ we have
    \begin{equation*}
        \frac{1}{\sinh r_i}\frac{dr_i}{dt} \leq \pi^{p-1}\left(\sum_{v_j \sim v_i}B_{ij}(1-\frac{\cosh l_{ij}-1}{\pi^{p-2}})\right) \leq 0,
    \end{equation*}
    hence we obtain \(r_i(t) \leq \max\{R,l,r_{i}(0)\}\). \par
    In summary, we prove both the lower and upper boundeness of solutions to the combinatorial $p$-th Calabi flow (\ref{Equation:combinatorial_p_th_calabi_flow}) in any finite time, which implies that the long time existence of
solutions to the flow (\ref{Equation:combinatorial_p_th_calabi_flow}). This completes the proof. \hfill $\square$

\section{Appendix}
\subsection{A proof of Lemma \ref{Lemma:Glickenstein-lemma}}

\begin{proof}
    We now prove the second equality. To establish the first equality in the equality (\ref{Equation:Glickenstein-lemma}), it suffices to show that
    \begin{equation*}
        \frac{\partial \theta_{i}^{jk}}{\partial u_{j}} = \frac{\partial \theta_{j}^{ik}}{\partial u_{i}}, \quad \frac{\partial \theta_{i}^{jk}}{\partial u_{k}} = \frac{\partial \theta_{k}^{ij}}{\partial u_{i}},
    \end{equation*}
    which follows directly from Lemma A1 in Chow-Luo \cite{BennettChow2003}. However, we also derive these relations in the subsequent proof and discuss them in Remark \ref{Remark:Glickenstein-lemma}.
    For convenience, we adopt the following notations: $\theta_{i} = \theta_{i}^{jk},\ C_{ij} = \cosh l_{ij},\ S_{ij} = \cosh l_{ij},\ C_{i} = \cosh r_{i},\ S_{i} = \sinh r_{i},\ A_{ijk} = \sinh l_{ij}\sinh l_{ik}\sin \theta_{i}^{jk}.$ According to the hyperbolic cosine law, we have
    \begin{equation*}
    S_{ij} S_{ik}\cos \theta_{i} = C_{ij} C_{ik} - C_{jk}.
    \end{equation*}
    Hence, we compute the partial derivatives of both sides of the above equality with respect to $l_{ij}$, $l_{ik}$, and $l_{jk}$ as follows
    \begin{equation*}
        \begin{gathered}
            -\frac{\partial \theta_i}{\partial l_{ij}} S_{ij}S_{ik}\sin \theta_i + C_{ij}S_{ik} \cos \theta_i = S_{ij}C_{ik},
            \\
            -\frac{\partial \theta_i}{\partial l_{ik}} S_{ij}S_{ik}\sin \theta_i + S_{ij}C_{ik} \cos \theta_i = C_{ij}S_{ik},
            \\
            -\frac{\partial \theta_i}{\partial l_{jk}}S_{ij}S_{ik} \sin \theta_i = -S_{jk}.
        \end{gathered}
    \end{equation*}
    A direct computation gives
    \begin{equation}
        \label{Equation:Glickenstein-lemma-proof_1}
        \begin{gathered}
            \frac{\partial \theta_i}{\partial l_{ij}} = - \frac{S_{ij}C_{ik}-C_{ij}S_{ik}\cos \theta_i}{A_{ijk}} = -\frac{S_{ij}C_{ik} - \frac{C_{ij}^2C_{ik}-C_{ij}C_{jk}}{S_{ij}}}{A_{ijk}}  = -\frac{S_{jk}\cos \theta_j}{A_{ijk}},\\
            \frac{\partial \theta_i}{\partial l_{ik}} = - \frac{C_{ij}S_{ik}-S_{ij}C_{ik}\cos \theta_i}{A_{ijk}} = - \frac{C_{ij}S_{ik} - \frac{C_{ij}C_{ik}^2 - C_{ik}C_{jk}}{S_{ik}}}{A_{ijk}} = -\frac{S_{jk}\cos \theta_k}{A_{ijk}}, \\
            \frac{\partial \theta_i}{\partial l_{jk}} = \frac{S_{jk}}{A_{ijk}}.
        \end{gathered}
    \end{equation}

    Next, recall the formula $C_{ij} = C_i C_j + \cos \Phi_{ij} S_i S_j$, and we obtain
    \begin{equation}
        \label{Equation:Glickenstein-lemma-proof_2}
        \begin{aligned}
            \frac{\partial l_{ij}}{\partial u_i} = S_{ij} \frac{\partial l_{ij}}{\partial r_i}
            &= \frac{\partial C_i}{\partial r_i} C_j + \cos \Phi_{ij} \frac{\partial S_i}{\partial r_i} S_j \\
            &= S_i C_j + \cos \Phi_{ij} C_i S_j, \\
            \frac{\partial l_{ik}}{\partial u_i} = S_{ik} \frac{\partial l_{ik}}{\partial r_i}
            &= \frac{\partial C_i}{\partial r_i} C_k + \cos \Phi_{ik} \frac{\partial S_i}{\partial r_i} S_k \\
            &= S_i C_k + \cos \Phi_{ik} C_i S_k, \\
            \frac{\partial l_{jk}}{\partial u_i} = S_{jk} \frac{\partial l_{jk}}{\partial r_i}
            &= 0.
        \end{aligned}
    \end{equation}

    Based on the chain rule and the results we have already computed in (\ref{Equation:Glickenstein-lemma-proof_1}) and (\ref{Equation:Glickenstein-lemma-proof_2}), we present the following expression for $A_{ijk} \frac{\partial \theta_i}{\partial u_i}$ as follows:
    \begin{equation}
        \label{Equation:Glickenstein-lemma-proof_3}
        \begin{aligned}
            A_{ijk}\frac{\partial \theta_i}{\partial u_i} &=
            A_{ijk} S_i \frac{\partial \theta_i}{\partial r_i}
            = A_{ijk} S_i (\frac{\partial \theta_i}{\partial l_{ij}} \frac{\partial l_{ij}}{\partial r_i} + \frac{\partial \theta_i}{\partial l_{ik}} \frac{\partial l_{ik}}{\partial r_i})
            \\
            &=A_{ijk}S_i(-\frac{S_{jk} \cos \theta_j}{A_{ijk}}\frac{S_iC_j + \cos \Phi_{ij} C_i S_j}{S_{ij}} - \frac{-S_{jk}\cos \theta_k}{A_{ijk}} \frac{S_iC_k + \cos \Phi_{ik} C_i S_k}{S_{ik}})
            \\
            &= -\frac{S_{jk} \cos \theta_j}{S_{ij}}(S_i^2 C_j + \cos \Phi_{ij} C_i S_i S_j) - \frac{S_{jk} \cos \theta_k}{S_{ik}} (S_i^2 C_k + \cos \Phi_{ik} C_i S_i S_k)
            \\
            &=  -\frac{S_{jk} \cos \theta_j}{S_{ij}}(C_i C_{ij} - C_j) - \frac{S_{jk} \cos \theta_k}{S_{ik}} (C_i C_{ik} - C_k)
            \\
            &= - \frac{S_{ik}^2(S_{ij}S_{jk}\cos \theta_j)(C_i C_{ij} - C_j) + S_{ij}^2(S_{ik}S_{jk}\cos \theta_k)(C_iC_{ik}-C_k)}{(S_{ij}S_{ik})^2}
            \\
            &= - \frac{S_{ik}^2(C_{ij}C_{jk} - C_{ik})(C_iC_{ij} - C_j) + S_{ij}^2(C_{ik}C_{jk} - C_{ij})(C_iC_{ik} - C_k)}{(S_{ij}S_{ik})^2}.
        \end{aligned}
    \end{equation}
    Similarly, we compute the expression for $A_{ijk}\,\frac{\partial \theta_i}{\partial u_j}$ as follows:
    \begin{equation}
        \label{Equation:Glickenstein-lemma-proof_4}
        \begin{aligned}
            A_{ijk}\frac{\partial \theta_i}{\partial u_j} &= A_{ijk}S_j \frac{\partial \theta_i}{\partial r_j}
            = A_{ijk}S_j(\frac{\partial \theta_i}{\partial l_{ij}} \frac{\partial l_{ij}}{\partial r_j} + \frac{\partial \theta_i}{\partial l_{jk}} \frac{\partial l_{jk}}{\partial r_j})
            \\
            &= A_{ijk} S_j (-\frac{S_{jk}\cos \theta_j}{A_{ijk}}\frac{C_i S_j + \cos \Phi_{ij} S_i C_j}{S_{ij}} + \frac{S_{jk}}{A_{ijk}}\frac{S_jC_k + \cos \Phi_{jk} C_j S_k}{S_{jk}})
            \\
            &= - \frac{S_{ij}S_{jk}\cos \theta_j(C_iS_j^2 + \cos \Phi_{ij}S_iS_jC_j)}{S_{ij}^2} + S_j^2C_k + \cos \Phi_{jk} S_j S_k C_j
            \\
            &= -\frac{(C_{ij}C_{jk} - C_{ik})(C_j C_{ij} - C_i)}{S_{ij}^2} + C_j C_{jk} - C_k.
        \end{aligned}
    \end{equation}
    Then we compute the expression for $A_{ijk}\,\frac{\partial \theta_i}{\partial u_k}$ as follows:
    \begin{equation}
        \label{Equation:Glickenstein-lemma-proof_5}
        \begin{aligned}
            A_{ijk}\frac{\partial \theta_i}{\partial u_k} &= A_{ijk} S_k \frac{\partial \theta_i}{\partial r_k}
            \\
            &= A_{ijk}S_k(\frac{\partial \theta_i}{\partial l_{ik}} \frac{\partial l_{ik}}{\partial r_k} + \frac{\partial \theta_i}{\partial l_{jk}} \frac{\partial l_{jk}}{\partial r_k})
            \\
            &= A_{ijk}S_k(-\frac{S_{jk}\cos \theta_k}{A_{ijk}}\frac{C_iS_k+\cos \Phi_{ik}S_iC_k}{S_{ik}}+\frac{S_{jk}}{A_{ijk}}\frac{C_jS_k+\cos \Phi_{jk}S_jC_k}{S_{jk}})
            \\
            &= -\frac{S_{ik}S_{jk}\cos \theta_k(C_iS_k^2 + \cos \Phi_{ik}S_iS_kC_k)}{S_{ik}^2}+S_k^2C_j + \cos \Phi_{jk}S_jS_kC_k
            \\
            &= -\frac{(C_{ik}C_{jk}-C_{ij})(C_kC_{ik}-C_i)}{S_{ik}^2} + C_kC_{jk}-C_j.
        \end{aligned}
    \end{equation}
    Next, based on the equalities (\ref{Equation:Glickenstein-lemma-proof_4}) and (\ref{Equation:Glickenstein-lemma-proof_5}), we provide the explicit expression for the right-hand side of the equality (\ref{Equation:Glickenstein-lemma}).
    \begin{equation}
        \label{Equation:Glickenstein-lemma-proof_6}
        \begin{aligned}
            A_{ijk} \frac{\partial \theta_i}{\partial u_j} C_{ij} &= (-\frac{(C_{ij}C_{jk} - C_{ik})(C_j C_{ij} - C_i)}{S_{ij}^2} + C_j C_{jk} - C_k)C_{ij}
            \\
            &= -\frac{(C_{ij}C_{jk} - C_{ik})(-C_i C_{ij} + C_j + C_j S_{ij}^2)}{S_{ij}^2} + C_j C_{jk} C_{ij} - C_k C_{ij}
            \\
            & = \frac{(C_{ij}C_{jk} - C_{ik})(C_i C_{ij} - C_j )}{S_{ij}^2} - C_{ij} C_{jk} C_j + C_{ik} C_j+ C_j C_{jk} C_{ij} - C_k C_{ij}
            \\
            &= \frac{(C_{ij}C_{jk} - C_{ik})(C_i C_{ij} - C_j )}{S_{ij}^2}+ C_{ik} C_j - C_k C_{ij},
        \end{aligned}
    \end{equation}
    \begin{equation}
        \label{Equation:Glickenstein-lemma-proof_7}
        \begin{aligned}
            A_{ijk} \frac{\partial \theta_i}{\partial u_k} C_{ik} &= (-\frac{(C_{ik}C_{jk} - C_{ij})(C_k C_{ik} - C_i)}{S_{ik}^2} + C_k C_{jk} - C_j)(C_{ik})
            \\
            &= -\frac{(C_{ik}C_{jk} - C_{ij})(- C_iC_{ik} + C_k + C_k S_{ik}^2 )}{S_{ik}^2} + C_k C_{jk} C_{ik} - C_j C_{ik}
            \\
            &=  \frac{(C_{ik}C_{jk} - C_{ij})(C_iC_{ik} - C_k )}{S_{ik}^2} - C_{ik} C_{jk} C_k + C_{ij}C_k + C_k C_{jk} C_{ik} - C_{j} C_{ik}
            \\
            &= \frac{(C_{ik}C_{jk} - C_{ij})(C_iC_{ik} - C_k )}{S_{ik}^2}+ C_{ij}C_k- C_{j} C_{ik},
        \end{aligned}
    \end{equation}
    Finally, combining the equalities (\ref{Equation:Glickenstein-lemma-proof_3}), (\ref{Equation:Glickenstein-lemma-proof_6}) and (\ref{Equation:Glickenstein-lemma-proof_7}), we have
    \begin{equation*}
        \begin{aligned}
            &\quad A_{ijk}(\frac{\partial \theta_i}{\partial u_k}C_{ik} + \frac{\partial \theta_i}{\partial u_j}C_{ij} + \frac{\partial \theta_i}{\partial u_i})
            \\
            &=\frac{(C_{ik}C_{jk} - C_{ij})(C_iC_{ik} - C_k )}{S_{ik}^2}+ C_{ij}C_k- C_{j} C_{ik}+\frac{(C_{ij}C_{jk} - C_{ik})(C_i C_{ij} - C_j )}{S_{ij}^2}+ C_{ik} C_j - C_k C_{ij}
            \\
            &\quad - \frac{S_{ik}^2(C_{ij}C_{jk} - C_{ik})(C_iC_{ij} - C_j) + S_{ij}^2(C_{ik}C_{jk} - C_{ij})(C_iC_{ik} - C_k)}{(S_{ij}S_{ik})^2}
            \\
            &= \frac{(1-1)S_{ik}^2(C_{ij}C_{jk} - C_{ik})(C_iC_{ij} - C_j) + (1-1)S_{ij}^2(C_{ik}C_{jk} - C_{ij})(C_iC_{ik} - C_k)}{(S_{ij}S_{ik})^2}
            \\
            &= 0,
        \end{aligned}
    \end{equation*}
    which implies that $\frac{\partial \theta_i}{\partial u_k}C_{ik} + \frac{\partial \theta_i}{\partial u_j}C_{ij} + \frac{\partial \theta_i}{\partial u_i}=0$.
    This completes the proof of Lemma \ref{Lemma:Glickenstein-lemma}. \hfill $\square$
\end{proof}


\begin{thebibliography}{50}
    \setlength{\itemsep}{-2pt}\small


    \bibitem{Angle2016} O. Angle, T. Hutchcroft, A. Nachmias, G. Ray, {\emph{Unimodular hyperbolic triangulation: circle packing and random walk}}, Invent. Math. 206 (2016) 229-268.

    \bibitem{Beardon1990} A. Beardon, K. Stephenson, {\emph{Circle packings in different geometries}}, Tohoku Math. J. 43 (1991) 27-36.

    \bibitem{Calabi1982} E. Calabi, {\emph{Extremal Kähler metrics}}, Ann. of Math. Stud., vol. 102, Princeton Univ. Press, Princeton, N.J., 1982, pp. 259-290.

    \bibitem{Calabi1985} E. Calabi, {\emph{Extremal Kähler metrics II}}, Differential geometry and complex analysis, Springer, Berlin, (1985) 95-114.

    \bibitem{Chang2000} S. Chang, {\emph{Global existence and convergence of solutions of Calabi flow on surfaces of genus h $\geq$ 2}}, Kyoto J. Math. 40 (2000) 363-377.

    \bibitem{Chang2006} S. Chang, {\emph{The 2-dimensional Calabi flow}}, Nagoya Math. J. 181 (2006) 63-73.

    \bibitem{Chen2001} X. Chen, {\emph{Calabi flow in Riemann surfaces revisited: a new point of view}}, Int. Math. Res. Not. (2001) 275-297 .

    \bibitem{BennettChow2003} B. Chow, F. Luo, {\emph{Combinatorial Ricci flows on surfaces}}, J. Differ. Geom. 63 (2003) 97-129.

    \bibitem{Feng22} K. Feng, H. Ge, B. Hua, {\emph{Combinatorial Ricci flows and the hyperbolization of a class of compact 3-manifolds}}, Geom. Topol. 26 (2022) 1349-1384.

    \bibitem{Feng2020} K. Feng, A. Lin, X. Zhang, {\emph{Combinatorial $p$-th Calabi flows for discrete conformal factors on surfaces}} J. Geom. Anal. 30 (2020) 3979-3994.

    \bibitem{Ge-thesis} H. Ge, \emph{Combinatorial methods and geometric equations}, Thesis (Ph.D.)-Peking University, Beijing. (2012). 144 pp.

    \bibitem{Ge2017a} H. Ge, {\emph{Combinatorial Calabi flows on surfaces}}, Trans. Am. Math. Soc. 370 (2017) 1377-1391.

    \bibitem{Ge2018} H. Ge, B. Hua, \emph{On combinatorial calabi flow with hyperbolic circle patterns}, Adv. Math. 333 (2018), 523-538.

    \bibitem{Ge2019} H. Ge, W. Jiang, {\emph{On the deformation of inversive distance circle packings I}}, Trans. Am. Math. Soc. 372 (2019) 6231-6261.

    \bibitem{Ge2019a} H. Ge, W. Jiang, {\emph{On the deformation of inversive distance circle packings II}}, J. Funct. Anal. 272 (2017) 3573-3595.

    \bibitem{Ge2017} H. Ge, W. Jiang, {\emph{On the deformation of inversive distance circle packings III}}, J. Funct. Anal. 272 (2017) 3573-3595.

    \bibitem{Ge2024} H. Ge, A. Lin, {\emph{The character of Thurston's circle packings}},  Sci. China. Math. 67 (2024) 1869-1862.

    \bibitem{Ge2016} H. Ge, X. Xu, \emph{2-dimensional combinatorial calabi flow in hyperbolic background geometry},
    Differ. Geom. Appl.  47 (2016), 86-98.

    \bibitem{Glickenstein2017} D. Glickenstein, J. Thomas, \emph{Duality structures and discrete conformal variations of piecewise constant curvature surfaces},
    Adv. Math. 320 (2017), 250-278.

    \bibitem{Guo2011} R. Guo, {\emph{Combinatorial Yamabe flow on hyperbolic surfaces with boundary}}, Comm. Contemp. Math. 13 (2011) 827-842.

    \bibitem{Hamilton1982} R. Hamilton, {\emph{Three-manifolds with positive Ricci curvature}}, J. Differ. Geom. 17 (1982) 255-306.

    \bibitem{Hu2024} G. Hu, Z. Lei, Y. Qi, P. Zhou, {\emph{Combinatorial p-th Calabi Flows for Total Geodesic Curvatures in Hyperbolic Background Geometry}},  J. Geom. Anal. 35 (2025) 1-18.

    \bibitem{He1995} Z. He, O. Schramm, {\emph{Hyperbolic and parabolic circle packings}}, Discrete Comput. Geom. 14 (1995) 123-149.

    \bibitem{Lin2019} A. Lin, X. Zhang, {\emph{Combinatorial $p$-th Calabi flows on surfaces}}, Adv. Math. 346 (2019) 1067-1090.

    \bibitem{Lin2021} A. Lin, X. Zhang, {\emph{Combinatorial $p$-th Ricci flows on surfaces}}, Nonlinear Anal. 211 (2021) 112417, 12 pp.

    \bibitem{Luo2004} F. Luo, {\emph{Combinatorial Yamabe flow on surfaces}}, Commun. Contemp. Math. 6 (2004) 765-780.

    \bibitem{Luo2005} F. Luo, {\emph{A combinatorial curvature flow for compact 3-manifolds with boundary}}, Electron. Res. Announc. Am. Math. Soc. 11 (2005) 12-20.

    \bibitem{Thurston1976} W. Thurston, \emph{Geometry and Topology of 3-Manifolds},
    Princeton Lecture Notes, 1976.

    \bibitem{Verdiere1991} Y. Colin de Verdière, {\emph{Un principe variationnel pour les empilements de cercles}}, Invent. Math. 104 (3) (1991) 655-669.

    \bibitem{Xu2018} X. Xu, \emph{Rigidity of inversive distance circle packings revisited},
    Adv. Math. 332 (2018), 476-509.

    \bibitem{Xu2021} T. Wu, X. Xu, {\emph{Fractional combinatorial Calabi flow on surfaces}}. arXiv:2107.14102.

    \bibitem{Zhang24} X. Zhang, X. Zheng, {\emph{On the deformation of Thurston's circle packings with obtuse intersection angles}}, J. Geom. Anal. 264 (2024) 1-24.

    \bibitem{Zhou2019} Z. Zhou, \emph{Circle patterns with obtuse exterior intersection angles}.
    arXiv:1703.01768.
\end{thebibliography}
\end{document}